\documentclass[3p,10pt,nopreprintline]{elsarticle}

\usepackage{amsmath}
\usepackage{amsthm}
\usepackage{amssymb}
\usepackage{graphicx}

\usepackage{epsfig}
\usepackage[normal]{subfigure}
\usepackage{float}
\usepackage{paralist}
\usepackage{caption2}
\usepackage{hyperref}
\usepackage{aliascnt}
\usepackage[T1]{fontenc}
\usepackage{natbib}
\usepackage{CJK}
\usepackage{color}
\usepackage[table]{xcolor}
\usepackage{colortbl,booktabs}

\newtheoremstyle{theorem}{6pt}{6pt}{\itshape}{}{\bfseries}{.}{.5em}{}
\newtheoremstyle{definition}{6pt}{6pt}{\upshape}{}{\bfseries}{.}{.5em}{}

\theoremstyle{theorem}
\newtheorem{theorem}{Theorem}[section]

\allowdisplaybreaks[4]

\newaliascnt{corollary}{theorem}

\aliascntresetthe{corollary}

\newaliascnt{lemma}{theorem}
\newtheorem{lemma}[lemma]{Lemma}
\aliascntresetthe{lemma}

\newaliascnt{sublemma}{theorem}

\aliascntresetthe{sublemma}

\theoremstyle{definition}

\newaliascnt{proposition}{theorem}

\aliascntresetthe{proposition}

\newcommand{\R}{{\mathbb{R}}}

\newcommand{\D}{{\nabla}}
\newcommand{\dif}{{\mathrm d}}

\numberwithin{equation}{section}

\begin{document}

\begin{frontmatter}

\title{Global well-posedness and large-time behavior for three-dimensional magneto-micropolar equations with horizontal dissipation}

\author[label1]{Peng Lu}
\address[label1]{Department of Mathematics, Zhejiang Sci-Tech University, Hangzhou, Zhejiang 310018, P.R.China;}
\ead{plu25@zstu.edu.cn}

\author[label2]{Yuanyuan Qiao\corref{cor1}}
\address[label2]{School of Mathematics and Information Science, Henan Polytechnic University, Jiaozuo, Henan 454000, P.R.China;}
\cortext[cor1]{Corresponding author.}
\ead{yyqiao@hpu.edu.cn}

\begin{abstract}
This paper is concerned with the stability and large-time behavior for 3D magneto-micropolar equations with horizontal dissipation. The global well-posedness of the aforementioned system is established, with the initial data and its vertical derivatives required to be sufficiently small in $L^2$ space. Moreover, we obtain the optimal decay rates for the $H^1$-norm of the solution. The proofs of our main results rely on the special structure of the equations and anisotropic Sobolev-type inequalities.
\end{abstract}


\end{frontmatter}

\section{Introduction}
The magneto-micropolar equations were introduced in \cite{MR443550} to describe the motion of an incompressible, electrically conducting micropolar fluids in the presence of an arbitrary magnetic field, such as salt water, ester, fluorocarbon, etc., which is of great importance in practical and mathematical applications. The three-dimensional viscous and resistive magneto-micropolar system is written as
\begin{equation}\label{0.0}
\begin{cases}
u_t+(u\cdot\D)u-\mu\Delta u-\chi\Delta u+\D P=(B\cdot\D)B+2\chi\D\times w, \\
B_t+(u\cdot\D)B-\nu\Delta B=(B\cdot\D)u, \\
w_t+(u\cdot\D)w-\gamma\Delta w-\kappa\nabla(\nabla\cdot w)+4\chi w=2\chi\D\times u, \\
\D\cdot u=\D\cdot B=0,
\end{cases}
\end{equation}
where $u=(u_1, u_2, u_3)$ is the velocity, $B=(B_1, B_2, B_3)$ is the magnetic field, $w=(w_1, w_2, w_3)$ is the micro-rotational velocity, and $P$ is the pressure. $\mu$, $\chi$ and $\nu$ are positive constants, representing the kinematic viscosity, the vortex viscosity, and the magnetic diffusivity. The positive constants $\gamma$ and $\kappa$ are the spin viscosity.

As is pointed out in \cite{MR2300252, MR2842964}, in geophysical fluids, meteorologist often modelize turbulent diffusion by putting a viscosity of the form $\mu_h(\partial_1^2+\partial_2^2)+\mu_3\partial_3^2$ instead of the classical viscosity $\mu\Delta$, where $\partial_j=\frac{\partial}{\partial x_j}$ for $j=1,2,3$, and $\mu_3$ is usually much smaller than $\mu_h$. We refer to J. Pedlovsky \cite[Chapter 4]{PJ} for a more complete discussion. In this paper, we focus on the following Cauchy problem of magneto-micropolar equations with horizontal dissipation.
\begin{equation}\label{BMHD}
\begin{cases}
u_t+(u\cdot\D)u-\mu\Delta_h u-\chi\Delta u+\D P=(B\cdot\D)B+2\chi\D\times w, \\
B_t+(u\cdot\D)B-\nu\Delta_h B=(B\cdot\D)u, \\
w_t+(u\cdot\D)w-\gamma\Delta_h w-\kappa\nabla(\nabla\cdot w)+4\chi w=2\chi\D\times u, \\
\D\cdot u=\D\cdot B=0,
\end{cases}
\end{equation}
where $x=(x_1,x_2, x_3)\in\R^3$ and $t\geq 0$ denote the spatial coordinate and time coordinate, and $\Delta_h=\partial_1^2+\partial_2^2$ denotes the horizontal Laplacian. The initial data is given by
\begin{equation}\label{data}
(u, B, w)(0,x)=(u_0, B_0, w_0)(x),
\end{equation}
satisfying $\D\cdot u_0=0$ and $\D\cdot B_0=0$ in the sense of distribution.

The well-posedness problem and large time behavior for the three-dimensional magneto-micropolar equations have attracted considerable attention in recent decades. The local existence and uniqueness of strong solutions, the global existence of strong solutions for small initial data, and the global existence of weak solutions were established in \cite{MR1810322, MR1484679, MR1666509}, respectively. In 2017, Wang and Wang \cite{MR3543126} considered the equations with mixed partial viscosity, and established the global existence of smooth solutions with initial data sufficiently small in $H^1$. In 2018, Li and Shang \cite{MR3825173} established the optimal decay rates of global weak and smooth solutions to \eqref{0.0} with small data. In 2019, Tan, Wu and Zhou \cite{MR3912713} established the global well-posedness of solutions in $H^s$ for $s\geq 3$, and obtained the decay rate of the solutions with initial data belongs to $\dot{H}^{-s}$ or $\dot{B}^{-s}_{2, \infty}$ with $0<s<\frac{3}{2}$. In 2021, Wang and Li \cite{MR4198768} proved the global well-posedness of the solutions in $H^3$ with mixed partial viscosity near a background magnetic field, where the authors only need one directional viscosity in the equations of $B$. In 2022, Jia, Xie and Dong \cite{MR4346494} established the global existence of smooth solutions to \eqref{0.0} with fractional dissipation $(-\Delta)^\alpha$. Recently, Shang and Liu \cite{MR4848928} established the global well-posedness and optimal decay estimates for smooth solutions in Besov spaces. For more related results, we refer to \cite{MR4915057, MR4723823, MR4119558, MR4717812, MR4204713, MR4904435, MR4853431, MR4653394, MR4108623} and the references cited therein.

When the magnetic field $B=0$, \eqref{0.0} reduces to the three-dimensional micropolar equations, which describe physical phenomena such as the motion of animal blood, liquid crystals, and dilute aqueous polymer solutions. The micropolar fluid model was first proposed by Eringen \cite{MR204005}. The existence of weak and strong solutions was investigated by Galdi and Rionero \cite{MR467030} and Yamaguchi \cite{MR2158216}, respectively. In recent years, several important results have been obtained regarding the stability and long-time behavior of solutions to the micropolar equations. In 2012, Chen and Miao \cite{MR2860636} proved the global well-posedness of the solutions in critical Besov spaces. In 2020, Wang, Wu and Ye \cite{MR4085363} established the global well-posedness of the smooth solutions to the equations with fractional dissipation. In 2021, Remond-Tiedrez and Tice \cite{MR4218675} studied the equations with anisotropic microstructure in a periodic domain, and proved the nonlinearly instability of a nontrivial equilibrium. In 2022, Li and Xiao \cite{MR4357379} studied the large-time behavior of the equations with a nonlinear velocity damping term $|u|^{\beta-1}u$ for $\beta>\frac{14}{5}$. Ye, Wang and Jia \cite{MR4379344} proved the global well-posedness and decay rates for the equations with only velocity dissipation. In 2024, Song \cite{MR4748113} proved the global well-posedness and decay rates of the solutions in critical Besov spaces. For more results in two-dimensional spaces, we refer to \cite{ MR3592648, MR3814368, MR2644133, MR4483343, MR3694625} and the references therein.

Recently, Shang and Liu \cite{MR4850530} proved the global well-posedness and optimal time decay of solutions to 3D micropolar equations with horizontal dissipation. They imposed a smallness assumption solely on the $L^2$-norm of the initial data and its $x_3$-directional derivative, which in turn effectively compensates for the lack of dissipation along the $x_3$-direction. Inspired by \cite{MR4850530}, we consider the global well-posedness and large-time behavior of three-dimensional magneto-micropolar equations with horizontal dissipation. Throughout this paper, for convenience, we denote
$$\int f\dif x=\int_{\mathbb{R}^3}f\dif x, \quad \Lambda^sf=\mathcal{F}^{-1}\left(|\xi|^s\hat{f}(\xi)\right), \quad \Lambda_h^sf=\mathcal{F}^{-1}\left(|\xi_h|^s\hat{f}(\xi)\right),$$
for $s\in\mathbb{R}$ and any function $f$, where $\mathcal{F}(f)$ or $\hat{f}$ denote the Fourier transformation of $f$. Moreover, we denote
\begin{align*}
& L^p=L^p(\mathbb{R}^3), ~ W^{k,p}=W^{k,p}(\mathbb{R}^3), ~ H^k=W^{k,2}, \\
& \dot{H}^{-\sigma}_h=\left\{f\in\mathcal{S}'(\mathbb{R}^3) ~ \big| ~ \Lambda_h^{-\sigma}f\in L^2\right\}, \quad\text{with} \quad \|f\|_{\dot{H}^{-\sigma}_h}=\|\Lambda_h^{-\sigma}f\|_{L^2}.
\end{align*}
We also write $\|f\|_{L^q_{x_j}}$ with $j =1,2,3$ for the $L^q$-norm with respect to $x_j$ on $\mathbb{R}$, and $\|f\|_{L^q_h}$ for the $L^q$-norm with respect to $(x_1, x_2)$ on $\mathbb{R}^2$.

Our main results can be stated as follows.
\begin{theorem}\label{th}
Suppose that $(u_0, B_0, w_0)\in H^1$ with $\D\cdot u_0=0$ and $\D\cdot B_0=0$. Then there exists a positive constant $\varepsilon>0$, such that if
\begin{equation}\label{th.1}
\|(u_0, B_0, w_0)\|_{L^2}+\|(\partial_3u_0, \partial_3B_0, \partial_3w_0)\|_{L^2}\leq\varepsilon,
\end{equation}
then \eqref{BMHD} admits a unique global solution $(u, B, w)$ satisfying
\begin{equation}\label{th.2}
\sup\limits_{t\geq 0}\Big(\|(u, B, w)\|_{L^2}+\|(\partial_3u, \partial_3B, \partial_3w )\|_{L^2}\Big)\leq C\varepsilon,
\end{equation}
and
\begin{equation}\label{th.3}
\sup\limits_{t\geq 0}\|(u, B, w)\|_{H^1}^2+\int_0^\infty\|(\nabla_hu, \nabla_hB, \nabla_hw)\|_{H^1}^2\dif t\leq C\|(u_0, B_0, w_0)\|_{H^1}^2.
\end{equation}

Moreover, if $(u_0, B_0, w_0)\in H^k$ for some $k\geq 2$, then the above global solution 
$(u, B, w)$ satisfies
\begin{align}\label{th.4}
&\|(u, B, w)(t)\|_{H^k}^2+\int_0^t\|(\nabla_hu, \nabla_hB, \nabla_hw)\|_{H^k}^2\dif\tau\nonumber\\
\leq~& C\|(u_0, B_0, w_0)\|_{H^k}^2\exp\left\{C\int_0^t\|(\nabla_hu, \nabla_hB, \nabla_hw)\|_{H^{k-1}}^2\dif\tau\right\},
\end{align}
for any $t\geq 0$, where $\nabla_h=(\partial_1, \partial_2)$.
\end{theorem}
\begin{theorem}\label{th2}
Suppose that $(u_0, B_0, w_0)\in H^k$ for $k\geq 4$ with \eqref{th.1}, and
\begin{equation}\label{th2.0}
(u_0, B_0, w_0, \partial_3u_0, \partial_3B_0, \partial_3w_0)\in\dot{H}_h^{-\sigma},
\end{equation}
for some $\frac{k-1}{2(k-2)}<\sigma<1$. Then the global solution $(u, B, w)$ of \eqref{BMHD} satisfies
\begin{align}\label{th2.1}
\big\|(u, B, w, \partial_3u, \partial_3B, \partial_3w)(t)\big\|_{\dot{H}_h^{-\sigma}}\leq~& C, \nonumber\\
\big\|(u, B, w, \partial_3u, \partial_3B, \partial_3w)(t)\big\|_{L^2}^2\leq~& C(1+t)^{-\sigma}, \nonumber\\
\big\|(\nabla_hu, \nabla_hB, \nabla_hw)(t)\big\|_{L^2}^2\leq~& C(1+t)^{-(1+\sigma)},
\end{align}
for any $t>0$.
\end{theorem}

Now we outline the proof of our main results. The first part is to prove the stability of the solution. Denote $U=(u, B, w)^\top$. A key observation from \cite{MR4850530} is that, by making a suitable linear combination of the \textit{a priori} estimates of $\|U\|_{L^2}$, $\|\partial_3U\|_{L^2}$ and $\|\D_hU\|_{L^2}$, we are able to derive estimates of the following type
\begin{align*}
\frac{\dif}{\dif t}\|(U, \partial_3U)\|_{L^2}^2+C\|(\D_hU, \D_h\partial_3U)\|_{L^2}^2\leq~& C\|\partial_3U\|_{L^2}\|(\D_hU, \D_h\partial_3U)\|_{L^2}^2, \\
\frac{\dif}{\dif t}\|U\|_{H^1}^2+C\|\D_hU\|_{H^1}^2\leq~& C\|(U, \partial_3U)\|_{L^2}\|\D_hU\|_{H^1}^2,
\end{align*}
see \eqref{1.7} and \eqref{1.18} for details. Therefore, through a standard bootstrap argument, one finds that we only need the $L^2$-norm of the initial data and its vertical derivatives to be small to establish the $H^1$-stability of the solution. This improves the corresponding result in \cite{MR4703479}, where the authors need the initial data to be sufficiently small in $H^1$-space to obtain the stability of 3D anisotropic MHD equations. Next, through an inductive process, we further prove that the $H^k$-norm of the solution is uniformly bounded via its $H^{k-1}$-norm. As a result, we establish the $H^k$-stability of the solution without assuming the higher-order derivatives of the initial data to be small.

The second part is to establish decay estimates of the solution. Inspired by \cite{MR4252141}, we make the ansatz that the following estimate holds for any $t\in[0,T]$,
\begin{equation*}
\big\|(\Lambda_h^{-\sigma}U, \Lambda_h^{-\sigma}\partial_3U)(t)\big\|_{L^2}^2\leq 2E_0,
\end{equation*}
where $T>0$ is a finite time guaranteed by the local well-posedness theory, and
\begin{equation*}
E_0:=\big\|(\Lambda_h^{-\sigma}U_0, \Lambda_h^{-\sigma}\partial_3U_0)\big\|_{L^2}^2<\infty.
\end{equation*}
Here we note that we do not need $E_0$ to be small. With the help of the above assumption, we are able to derive some decay estimates of $\|U\|_{L^2}$, $\|\partial_3U\|_{L^2}$ and $\|\D_hU\|_{L^2}$ on the time interval $[0,T]$, see Lemma \ref{lem.4} and Lemma \ref{lem.5} for details. Next, by invoking various anisotropic inequalities (see \eqref{6.2}), we obtain a better upper bound of $\|(\Lambda_h^{-\sigma}U, \Lambda_h^{-\sigma}\partial_3U)\|_{L^2}$. Therefore, a standard bootstrap argument implies that the decay estimates hold for any $t>0$.

The rest of this paper is organized as follows. In \S \ref{prelim}, we give some notations and preliminary lemmas which will be used frequently in our following proof. In \S \ref{Section.3} and \S \ref{Section.4}, we prove the main results of this paper.

\section{Preliminaries}\label{prelim}

In this section, we list some basic lemmas to be used later. The first lemma provides anisotropic estimates for the integral of a triple product (see \cite{MR4186009}).
\begin{lemma}\label{pre.1}
The following estimates hold when the right-hand sides are all bounded,
\begin{align*}
\int|fgh|\dif x\leq~&C\|f\|_{L^2}^{\frac{1}{2}}\|\partial_1f\|_{L^2}^{\frac{1}{2}}\|g\|_{L^2}^{\frac{1}{2}}\|\partial_2g\|_{L^2}^{\frac{1}{2}}\|h\|_{L^2}^{\frac{1}{2}}\|\partial_3h\|_{L^2}^{\frac{1}{2}}, \\
\int|fgh|\dif x\leq~&C\|f\|_{L^2}^{\frac{1}{4}}\|\partial_1f\|_{L^2}^{\frac{1}{4}}\|\partial_2f\|_{L^2}^{\frac{1}{4}}\|\partial_1\partial_2f\|_{L^2}^{\frac{1}{4}}\|g\|_{L^2}^{\frac{1}{2}}\|\partial_3g\|_{L^2}^{\frac{1}{2}}\|h\|_{L^2}.
\end{align*}
\end{lemma}

The second lemma provides an estimate for the $L^q$-norm of a one-dimensional function, which serves as a basic ingredient for anisotropic estimates (see \cite{MR3994306}).
\begin{lemma}\label{pre.2}
Let $q\in[2, \infty]$ and $s>\frac{1}{2}-\frac{1}{q}$. Then there exists a positive constant $C$ such that for any $f\in H^s(\mathbb{R})$, it holds that
\begin{equation*}
\|f\|_{L^q(\mathbb{R})}\leq C\|f\|_{L^2(\R)}^{1-\frac{1}{s}(\frac{1}{2}-\frac{1}{q})}\|\Lambda^sf\|_{L^2(\mathbb{R})}^{\frac{1}{s}(\frac{1}{2}-\frac{1}{q})}.
\end{equation*}
In particular, if $q=\infty$ and $s=1$, then any $f\in H^1(\R)$ satisfies
\begin{equation*}
\|f\|_{L^\infty(\mathbb{R})}\leq C\|f\|_{L^2(\R)}^{\frac{1}{2}}\|f'\|_{L^2(\mathbb{R})}^{\frac{1}{2}}.
\end{equation*}
\end{lemma}

\section{Proof of Theorem \ref{th}}\label{Section.3}

In this section, we shall prove Theorem \ref{th}. Since the local well-posedness of solutions in $H^s$ with $s\geq 1$ follows from standard arguments (see \cite{MR1867882} for instance), it suffices to establish the global \textit{a priori} estimates in $H^s$ for the solution. The proof is divided into two parts, where we establish the global existence of solutions in $H^1$ and $H^k$ for $k\geq 2$.

\subsection{Global existence of solutions in $H^1$}

In the following lemma, we shall prove the \textit{a priori} estimates of $\|(u, B, w)\|_{L^2}$ and $\|(\partial_3u, \partial_3B, \partial_3w)\|_{L^2}$.
\begin{lemma}
Suppose that $(u, B, w)$ is a solution to \eqref{BMHD} with initial data $(u_0, B_0, w_0)\in H^1$ satisfying $\nabla\cdot u_0=0$, $\nabla\cdot B_0=0$ and \eqref{th.1}. Then it holds that
\begin{equation*}
\sup\limits_{t\geq 0}\Big(\|(u, B, w)\|_{L^2}+\|(\partial_3u, \partial_3B, \partial_3w )\|_{L^2}\Big)\leq C\varepsilon.
\end{equation*}
\end{lemma}
\begin{proof}
Taking the $L^2$-inner product of \eqref{BMHD} with $(u, B, w)$, we obtain after integration by parts that
\begin{align}\label{1.0}
&\frac{1}{2}\frac{\dif}{\dif t}\|(u, B, w)\|_{L^2}^2+\mu\|\D_hu\|_{L^2}^2+\nu\|\D_hB\|_{L^2}^2+\gamma\|\D_hw\|_{L^2}^2+\kappa\|\nabla\cdot w\|_{L^2}^2+\chi\|\D u\|_{L^2}^2+4\chi\|w\|_{L^2}^2 \nonumber\\
=~&4\chi\int w\cdot(\D\times u)\dif x \nonumber\\
\leq~&\chi\|\D u\|_{L^2}^2+4\chi\|w\|_{L^2}^2,
\end{align}
which implies
\begin{equation}\label{1.1}
\frac{1}{2}\frac{\dif}{\dif t}\|(u, B, w)\|_{L^2}^2+\mu\|\D_hu\|_{L^2}^2+\nu\|\D_hB\|_{L^2}^2+\gamma\|\D_hw\|_{L^2}^2\leq 0
\end{equation}

Applying $\partial_3$ to \eqref{BMHD}, then taking the $L^2$-inner product with $(\partial_3u, \partial_3B, \partial_3w)$, we obtain
\begin{align}\label{1.2}
&\frac{1}{2}\frac{\dif}{\dif t}\|(\partial_3u, \partial_3B, \partial_3w)\|_{L^2}^2+\mu\|\D_h\partial_3u\|_{L^2}^2+\nu\|\D_h\partial_3B\|_{L^2}^2+\gamma\|\D_h\partial_3w\|_{L^2}^2+\chi\|\D\partial_3u\|_{L^2}^2+4\chi\|\partial_3w\|_{L^2}^2 \nonumber\\
\leq~&4\chi\int\partial_3w\cdot\partial_3(\D\times u)\dif x\int\partial_3(B\cdot\D B)\cdot\partial_3u\dif x-\int\partial_3(u\cdot\D u)\cdot\partial_3u\dif x \nonumber\\
&+\int\partial_3(B\cdot\D u)\cdot\partial_3B\dif x-\int\partial_3(u\cdot\D B)\cdot\partial_3B\dif x-\int\partial_3(u\cdot\D w)\cdot\partial_3w\dif x \nonumber\\
\leq~&\chi\|\D\partial_3u\|_{L^2}^2+4\chi\|\partial_3w\|_{L^2}^2+\int\partial_3(B\cdot\D B)\cdot\partial_3u\dif x-\int\partial_3(u\cdot\D u)\cdot\partial_3u\dif x \nonumber\\
&+\int\partial_3(B\cdot\D u)\cdot\partial_3B\dif x-\int\partial_3(u\cdot\D B)\cdot\partial_3B\dif x-\int\partial_3(u\cdot\D w)\cdot\partial_3w\dif x \nonumber\\
:=~&\chi\|\D\partial_3u\|_{L^2}^2+4\chi\|\partial_3w\|_{L^2}^2+\sum\limits_{i=1}^5I_i.
\end{align}

After integration by parts and using the equation $\D\cdot B=0$, we get
\begin{align}\label{1.3}
I_1+I_3=~&\int\partial_3B\cdot\D B\cdot\partial_3u\dif x+\int\partial_3B\cdot\D u\cdot\partial_3B\dif x+\int B\cdot\D(\partial_3u\cdot\partial_3B)\dif x \nonumber\\
=~&\int\partial_3B_h\cdot\D_hB\cdot\partial_3u\dif x+\int\partial_3B_3\partial_3B\cdot\partial_3u\dif x \nonumber\\
&+\int\partial_3B_h\cdot\D_hu\cdot\partial_3B\dif x+\int\partial_3B_3\partial_3u\cdot\partial_3B\dif x \nonumber\\
=~&\int\partial_3B_h\cdot\D_hB\cdot\partial_3u\dif x-\int\partial_1B_1\partial_3B\cdot\partial_3u\dif x-\int\partial_2B_2\partial_3B\cdot\partial_3u\dif x \nonumber\\
&+\int\partial_3B_h\cdot\D_hu\cdot\partial_3B\dif x-\int\partial_1B_1\partial_3u\cdot\partial_3B\dif x-\int\partial_2B_2\partial_3u\cdot\partial_3B\dif x,
\end{align}
where $B_h=(B_1, B_2)$. By Lemma \ref{pre.1} and Young inequality, we have
\begin{align}\label{1.4}
I_1+I_3\leq~&C\int|\partial_3B_h||\D_hB||\partial_3u|\dif x+C\int|\partial_hB_h||\partial_3B||\partial_3u|\dif x+C\int|\partial_3B_h||\D_hu||\partial_3B|\dif x \nonumber\\
\leq~&C\|\partial_3B_h\|_{L^2}^{\frac{1}{2}}\|\partial_1\partial_3B_h\|_{L^2}^{\frac{1}{2}}\|\D_hB\|_{L^2}^{\frac{1}{2}}\|\D_h\partial_3B\|_{L^2}^{\frac{1}{2}}\|\partial_3u\|_{L^2}^{\frac{1}{2}}\|\partial_2\partial_3u\|_{L^2}^{\frac{1}{2}} \nonumber\\
&+C\|\partial_hB_h\|_{L^2}^{\frac{1}{2}}\|\partial_h\partial_3B_h\|_{L^2}^{\frac{1}{2}}\|\partial_3B\|_{L^2}^{\frac{1}{2}}\|\partial_1\partial_3B\|_{L^2}^{\frac{1}{2}}\|\partial_3u\|_{L^2}^{\frac{1}{2}}\|\partial_2\partial_3u\|_{L^2}^{\frac{1}{2}} \nonumber\\
&+C\|\partial_3B_h\|_{L^2}^{\frac{1}{2}}\|\partial_1\partial_3B_h\|_{L^2}^{\frac{1}{2}}\|\D_hu\|_{L^2}^{\frac{1}{2}}\|\D_h\partial_3u\|_{L^2}^{\frac{1}{2}}\|\partial_3B\|_{L^2}^{\frac{1}{2}}\|\partial_2\partial_3B\|_{L^2}^{\frac{1}{2}} \nonumber\\
\leq~& C\|(\partial_3u, \partial_3B)\|_{L^2}\|(\D_hu, \D_hB, \D_h\partial_3u, \D_h\partial_3B)\|_{L^2}^2.
\end{align}

Similarly, using the equation $\D\cdot u=0$, we have
\begin{align}\label{1.5}
I_2+I_4+I_5=~&-\int\partial_3u\cdot\D u\cdot\partial_3u\dif x-\int\partial_3u\cdot\D B\cdot\partial_3B\dif x-\int\partial_3u\cdot\D w\cdot\partial_3w\dif x \nonumber\\
=~&-\int\partial_3u_h\cdot\D_hu\cdot\partial_3u\dif x-\int\partial_3u_3|\partial_3u|^2\dif x-\int\partial_3u_h\cdot\D_hB\cdot\partial_3B\dif x \nonumber\\
&-\int\partial_3u_3|\partial_3B|^2\dif x-\int\partial_3u_h\cdot\D_hw\cdot\partial_3w\dif x-\int\partial_3u_3|\partial_3w|^2\dif x \nonumber\\
=~&-\int\partial_3u_h\cdot\D_hu\cdot\partial_3u\dif x+\int\left(\partial_1u_1+\partial_2u_2\right)|\partial_3u|^2\dif x-\int\partial_3u_h\cdot\D_hB\cdot\partial_3B\dif x \nonumber\\
&+\int\left(\partial_1u_1+\partial_2u_2\right)|\partial_3B|^2\dif x-\int\partial_3u_h\cdot\D_hw\cdot\partial_3w\dif x+\int\left(\partial_1u_1+\partial_2u_2\right)|\partial_3w|^2\dif x \nonumber\\
\leq~&C\int|\partial_3u_h||\D_hu||\partial_3u|\dif x+C\int|\D_hu_h||\partial_3u|^2\dif x+C\int|\partial_3u_h||\D_hB||\partial_3B|\dif x \nonumber\\
&+C\int|\D_hu_h||\partial_3B|^2\dif x+C\int|\partial_3u_h||\D_hw||\partial_3w|\dif x+C\int|\D_hu_h||\partial_3w|^2\dif x,
\end{align}
where $u_h=(u_1, u_2)$ and $w_h=(w_1, w_2)$. By Lemma \ref{pre.1} and Young inequality, we have
\begin{align}\label{1.6}
I_2+I_4+I_5\leq~&C\|\partial_3u_h\|_{L^2}^{\frac{1}{2}}\|\partial_1\partial_3u_h\|_{L^2}^{\frac{1}{2}}\|\D_hu\|_{L^2}^{\frac{1}{2}}\|\D_h\partial_3u\|_{L^2}^{\frac{1}{2}}\|\partial_3u\|_{L^2}^{\frac{1}{2}}\|\partial_2\partial_3u\|_{L^2}^{\frac{1}{2}} \nonumber\\
&+C\|\D_hu_h\|_{L^2}^{\frac{1}{2}}\|\D_h\partial_3u_h\|_{L^2}^{\frac{1}{2}}\|\partial_3u\|_{L^2}\|\partial_1\partial_3u\|_{L^2}^{\frac{1}{2}}\|\partial_2\partial_3u\|_{L^2}^{\frac{1}{2}} \nonumber\\
&+C\|\partial_3u_h\|_{L^2}^{\frac{1}{2}}\|\partial_1\partial_3u_h\|_{L^2}^{\frac{1}{2}}\|\D_hB\|_{L^2}^{\frac{1}{2}}\|\D_h\partial_3B\|_{L^2}^{\frac{1}{2}}\|\partial_3B\|_{L^2}^{\frac{1}{2}}\|\partial_2\partial_3B\|_{L^2}^{\frac{1}{2}} \nonumber\\
&+C\|\D_hu_h\|_{L^2}^{\frac{1}{2}}\|\D_h\partial_3u_h\|_{L^2}^{\frac{1}{2}}\|\partial_3B\|_{L^2}\|\partial_1\partial_3B\|_{L^2}^{\frac{1}{2}}\|\partial_2\partial_3B\|_{L^2}^{\frac{1}{2}} \nonumber\\
&+C\|\partial_3u_h\|_{L^2}^{\frac{1}{2}}\|\partial_1\partial_3u_h\|_{L^2}^{\frac{1}{2}}\|\D_hw\|_{L^2}^{\frac{1}{2}}\|\D_h\partial_3w\|_{L^2}^{\frac{1}{2}}\|\partial_3w\|_{L^2}^{\frac{1}{2}}\|\partial_2\partial_3w\|_{L^2}^{\frac{1}{2}} \nonumber\\
&+C\|\D_hu_h\|_{L^2}^{\frac{1}{2}}\|\D_h\partial_3u_h\|_{L^2}^{\frac{1}{2}}\|\partial_3w\|_{L^2}\|\partial_1\partial_3w\|_{L^2}^{\frac{1}{2}}\|\partial_2\partial_3w\|_{L^2}^{\frac{1}{2}} \nonumber\\
\leq~& C\|(\partial_3u, \partial_3B, \partial_3w)\|_{L^2}\|(\D_hu, \D_hB, \D_hw, \D_h\partial_3u, \D_h\partial_3B, \D_h\partial_3 w)\|_{L^2}^2.
\end{align}

Inserting \eqref{1.4} and \eqref{1.6} into \eqref{1.2}, and adding the resulting inequality to \eqref{1.1}, we obtain
\begin{align}\label{1.7}
&\frac{\dif}{\dif t}\left\|\left(u, B, w, \partial_3u, \partial_3B, \partial_3w\right)\right\|_{L^2}^2+C_0\left\|\left(\D_hu, \D_hB, \D_hw, \D_h\partial_3u, \D_h\partial_3B, \D_h\partial_3w\right)\right\|_{L^2}^2 \nonumber\\
\leq~&C\|(\partial_3u, \partial_3B, \partial_3w)\|_{L^2}\|(\D_hu, \D_hB, \D_hw, \D_h\partial_3u, \D_h\partial_3B, \D_h\partial_3 w)\|_{L^2}^2,
\end{align}
for some positive constant $C_0$. Integrating \eqref{1.7} in $(0,t)$, and taking $\varepsilon$ in \eqref{th.1} sufficiently small, we obtain after standard bootstrap argument that
\begin{align}\label{1.8}
&\left\|\left(u, B, w, \partial_3u, \partial_3B, \partial_3w\right)(t)\right\|_{L^2}^2+\int_0^t\left\|\left(\D_hu, \D_hB, \D_hw, \D_h\partial_3u, \D_h\partial_3B, \D_h\partial_3w\right)\right\|_{L^2}^2\dif\tau \nonumber\\
\leq~&C\left\|\left(u_0, B_0, w_0, \partial_3u_0, \partial_3B_0, \partial_3w_0\right)\right\|_{L^2}^2,
\end{align}
holds for any $t\geq 0$. This completes the proof.
\end{proof}

The next lemma is devoted to establishing the a priori estimates of $\|(\D_hu, \D_hB, \D_hw)\|_{L^2}$.
\begin{lemma}
Suppose that $(u, B, w)$ is a solution to \eqref{BMHD} with initial data $(u_0, B_0, w_0)\in H^1$ satisfying $\nabla\cdot u_0=0$, $\nabla\cdot B_0=0$ and \eqref{th.1}. Then it holds that
\begin{equation*}
\sup\limits_{t\geq 0}\|(u, B, w)\|_{H^1}^2+\int_0^\infty\|(\nabla_hu, \nabla_hB, \nabla_hw)\|_{H^1}^2\dif t\leq C\|(u_0, B_0, w_0)\|_{H^1}^2.
\end{equation*}
\end{lemma}
\begin{proof}
Taking the $L^2$-inner product of \eqref{BMHD} with $(\Delta_hu, \Delta_hB, \Delta_hw)$, we obtain after integration by parts that
\begin{align}\label{1.9}
&\frac{1}{2}\frac{\dif}{\dif t}\left\|\left(\D_hu, \D_hB, \D_hw\right)\right\|_{L^2}^2+\mu\|\Delta_hu\|_{L^2}^2+\nu\|\Delta_hB\|_{L^2}^2+\gamma\|\Delta_hw\|_{L^2}^2+\chi\|\D\D_h u\|_{L^2}^2+4\chi\|\D_hw\|_{L^2}^2 \nonumber\\
\leq~&4\chi\int\D_hw\cdot\D_h(\D\times u)\dif x+\int u\cdot\D u\cdot\Delta_hu\dif x-\int B\cdot\D B\cdot\Delta_hu\dif x \nonumber\\
&+\int u\cdot\D B\cdot\Delta_hB\dif x-\int B\cdot\D u\cdot\Delta_hB\dif x+\int u\cdot\D w\cdot\Delta_hw\dif x \nonumber\\
\leq~&\chi\|\D\D_h u\|_{L^2}^2+4\chi\|\D_hw\|_{L^2}^2+\int u\cdot\D u\cdot\Delta_hu\dif x-\int B\cdot\D B\cdot\Delta_hu\dif x \nonumber\\
&+\int u\cdot\D B\cdot\Delta_hB\dif x-\int B\cdot\D u\cdot\Delta_hB\dif x+\int u\cdot\D w\cdot\Delta_hw\dif x \nonumber\\
:=~&\chi\|\D\D_h u\|_{L^2}^2+4\chi\|\D_hw\|_{L^2}^2+\sum\limits_{i=1}^5J_i.
\end{align}

Noting that $\nabla\cdot u=0$, we obtain after integration by parts that
\begin{align}\label{1.10}
J_1=~&-\sum\limits_{i,j=1}^3\sum\limits_{k=1,2}\int\partial_ku_j\partial_ju_i\partial_ku_i\dif x \nonumber\\
=~&-\int\partial_1u_h\cdot\D_hu\cdot\partial_1u\dif x-\int\partial_2u_h\cdot\D_hu\cdot\partial_2u\dif x \nonumber\\
&-\int\partial_1u_3\partial_3u\cdot\partial_1u\dif x-\int\partial_2u_3\partial_3u\cdot\partial_2u\dif x \nonumber\\
:=~&\sum\limits_{i=1}^4J_{1,i}.
\end{align}

From H\"older, Gagliardo-Nirenberg inequalities and Lemma \ref{pre.2}, we have
\begin{align}\label{1.11}
J_{1,1}+J_{1,2}\leq~&C\|\D_hu\|_{L^3}^3=C\big\|\|\D_hu\|_{L_h^3}\big\|_{L^3_{x_3}}^3 \nonumber\\
\leq~&C\Big\|\|u\|_{L_h^2}^{\frac{1}{3}}\|\Delta_hu\|_{L_h^2}^{\frac{2}{3}}\Big\|_{L^3_{x_3}}^3 \nonumber\\
\leq~&C\big\|\|u\|_{L^\infty_{x_3}}\big\|_{L_h^2}\|\Delta_hu\|_{L^2}^2 \nonumber\\
\leq~&C\Big\|\|u\|_{L^2_{x_3}}^{\frac{1}{2}}\|\partial_3u\|_{L^2_{x_3}}^{\frac{1}{2}}\Big\|_{L_h^2}\|\Delta_hu\|_{L^2}^2 \nonumber\\
\leq~&C\|u\|_{L^2}^{\frac{1}{2}}\|\partial_3u\|_{L^2}^{\frac{1}{2}}\|\Delta_hu\|_{L^2}^2.
\end{align}

From Lemma \ref{pre.1}, Young inequality and the equation $\nabla\cdot u=0$, we have
\begin{align}\label{1.12}
J_{1,3}+J_{1,4}\leq~&C\int|\partial_3u||\D_hu|^2\dif x \nonumber\\
\leq~&C\|\D_hu\|_{L^2}\|\partial_1\D_hu\|_{L^2}^{\frac{1}{2}}\|\partial_3u\|_{L^2}^{\frac{1}{2}}\|\partial_2\partial_3u\|_{L^2}^{\frac{1}{2}}\|\D_h\partial_3u\|_{L^2}^{\frac{1}{2}} \nonumber\\
\leq~&C\|u\|_{L^2}^{\frac{1}{2}}\|\Delta_hu\|_{L^2}^{\frac{1}{2}}\|\partial_1\D_hu\|_{L^2}^{\frac{1}{2}}\|\partial_3u\|_{L^2}^{\frac{1}{2}}\|\partial_2\partial_3u\|_{L^2}^{\frac{1}{2}}\|\D_h\partial_3u\|_{L^2}^{\frac{1}{2}} \nonumber\\
\leq~&C\|u\|_{L^2}^{\frac{1}{2}}\|\partial_3u\|_{L^2}^{\frac{1}{2}}\|\D_h\D u\|_{L^2}^2,
\end{align}
where we have used the fact
\begin{equation*}
\|\D_hu\|_{L^2}^2=\int\D_hu\cdot\D_hu\dif x=-\int u\cdot\Delta_hu\dif x\leq\|u\|_{L^2}\|\Delta_hu\|_{L^2}.
\end{equation*}
Inserting \eqref{1.11} and \eqref{1.12} into \eqref{1.10}, we obtain
\begin{equation}\label{1.13}
J_1\leq C\|u\|_{L^2}^{\frac{1}{2}}\|\partial_3u\|_{L^2}^{\frac{1}{2}}\|\D_h\D u\|_{L^2}^2.
\end{equation}

Similarly, we have
\begin{align}\label{1.14}
J_3+J_5=~&-\sum\limits_{i=1}^3\sum\limits_{j=1,2}\int\partial_ju_i\partial_iB\cdot\partial_jB\dif x-\sum\limits_{i=1}^3\sum\limits_{j=1,2}\int\partial_ju_i\partial_iw\cdot\partial_jw\dif x \nonumber\\
=~&-\int\partial_1u_h\cdot\D_hB\cdot\partial_1B\dif x-\int\partial_2u_h\cdot\D_hB\cdot\partial_2B\dif x-\int\partial_1u_3\partial_3B\cdot\partial_1B\dif x \nonumber\\
&-\int\partial_2u_3\partial_3B\cdot\partial_2B\dif x-\int\partial_1u_h\cdot\D_hw\cdot\partial_1w\dif x-\int\partial_2u_h\cdot\D_hw\cdot\partial_2w\dif x \nonumber\\
&-\int\partial_1u_3\partial_3w\cdot\partial_1w\dif x-\int\partial_2u_3\partial_3w\cdot\partial_2w\dif x \nonumber\\
\leq~&C\|(\D_hu, \D_hB, \D_hw)\|_{L^3}^3+C\int\Big(|\partial_3B||\D_hB|+|\partial_3w||\D_hw|\Big)|\D_hu|\dif x \nonumber\\
\leq~&C\|(u, B, w, \partial_3u, \partial_3B, \partial_3w)\|_{L^2}\|(\D_h\D u, \D_h\D B, \D_h\D w)\|_{L^2}^2.
\end{align}

Now it remains to estimate $J_2$ and $J_4$. Noting that $\nabla\cdot B=0$, we obtain after integration by parts that
\begin{align}\label{1.15}
J_2+J_4=~&\sum\limits_{i,j=1}^3\sum\limits_{k=1,2}\int\partial_ku_i\partial_jB_i\partial_kB_j\dif x+\sum\limits_{i,j=1}^3\sum\limits_{k=1,2}\int\partial_kB_i\partial_ju_i\partial_kB_j\dif x \nonumber\\
=~&\int\partial_1u\cdot\partial_3B\partial_1B_3\dif x+\int\partial_2u\cdot\partial_3B\partial_2B_3\dif x+\sum\limits_{i,j=1,2}\int\partial_iu\cdot\partial_jB\partial_iB_j\dif x \nonumber\\
&+\int\partial_1B\cdot\partial_3u\partial_1B_3\dif x+\int\partial_2B\cdot\partial_3u\partial_2B_3\dif x+\sum\limits_{i,j=1,2}\int\partial_iB\cdot\partial_ju\partial_iB_j\dif x \nonumber\\
\leq~&C\int|\D_hB|\big(|\D_hu|+|\D_hB|\big)\big(|\partial_3u|+|\partial_3B|\big)\dif x+C\|(\D_hu, \D_hB)\|_{L^3}^3.
\end{align}
Similar to \eqref{1.11} and \eqref{1.12}, we have
\begin{equation}\label{1.16}
J_2+J_4\leq C\|(u, B, \partial_3u, \partial_3B)\|_{L^2}\|(\D_h\D u, \D_h\D B)\|_{L^2}^2.
\end{equation}

Substituting \eqref{1.13}, \eqref{1.14} and \eqref{1.16} into \eqref{1.9}, we get
\begin{align}\label{1.17}
&\frac{\dif}{\dif t}\left\|\left(\D_hu, \D_hB, \D_hw\right)\right\|_{L^2}^2+C_0\left\|\left(\Delta_hu, \Delta_hB, \Delta_hw\right)\right\|_{L^2}^2 \nonumber\\
\leq~&C\|(u, B, w, \partial_3u, \partial_3B, \partial_3w)\|_{L^2}\|(\D_h\D u, \D_h\D B, \D_h\D w)\|_{L^2}^2.
\end{align}
for some positive constant $C_0$. Adding \eqref{1.17} to \eqref{1.7}, we get
\begin{align}\label{1.18}
&\frac{\dif}{\dif t}\left\|\left(u, B, w\right)\right\|_{H^1}^2+C_0\left\|\left(\D_hu, \D_hB, \D_hw\right)\right\|_{H^1}^2 \nonumber\\
\leq~&C\|(u, B, w, \partial_3u, \partial_3B, \partial_3w)\|_{L^2}\left\|\left(\D_hu, \D_hB, \D_hw\right)\right\|_{H^1}^2.
\end{align}
Integrating \eqref{1.18} in $(0,t)$, and taking $\varepsilon$ in \eqref{th.1} sufficiently small, we obtain after standard bootstrap argument that
\begin{align}\label{1.19}
\left\|\left(u, B, w\right)(t)\right\|_{H^1}^2+\int_0^t\left\|\left(\D_hu, \D_hB, \D_hw\right)\right\|_{H^1}^2\dif\tau\leq C\left\|\left(u_0, B_0, w_0\right)\right\|_{H^1}^2,
\end{align}
holds for any $t\geq 0$. This completes the proof.
\end{proof}

\subsection{Global existence of solutions in $H^k$}

\begin{lemma}
Given $k\geq 2$. Suppose that $(u, B, w)$ is a solution to \eqref{BMHD} with initial data $(u_0, B_0, w_0)\in H^k$ satisfying $\nabla\cdot u_0=0$, $\nabla\cdot B_0=0$ and \eqref{th.1}. Then it holds that
\begin{align}\label{2.0}
&\|(u, B, w)(t)\|_{H^k}^2+\int_0^t\|(\nabla_hu, \nabla_hB, \nabla_hw)\|_{H^k}^2\dif\tau\nonumber\\
\leq~& C\|(u_0, B_0, w_0)\|_{H^k}^2\exp\left\{C\int_0^t\|(\nabla_hu, \nabla_hB, \nabla_hw)\|_{H^{k-1}}^2\dif\tau\right\},
\end{align}
for any $t\geq 0$.
\end{lemma}
\begin{proof}
Applying $\partial_i^k$ ($i=1,2,3$) to \eqref{BMHD}, and taking the $L^2$-inner product with $(\partial_i^ku, \partial_i^kB, \partial_i^kw)$, we obtain
\begin{align}\label{2.1}
&\frac{1}{2}\frac{\dif}{\dif t}\sum\limits_{i=1}^3\left(\|\partial_i^ku\|_{L^2}^2+\|\partial_i^kB\|_{L^2}^2+\|\partial_i^kw\|_{L^2}^2\right)+\sum\limits_{i=1}^3\left(\chi\|\D\partial_i^ku\|_{L^2}^2+4\chi\|\partial_i^kw\|_{L^2}^2\right) \nonumber\\
&+\sum\limits_{i=1}^3\left(\mu\|\partial_i^k\D_hu\|_{L^2}^2+\nu\|\partial_i^k\D_hB\|_{L^2}^2+\gamma\|\partial_i^k\D_hw\|_{L^2}^2\right) \nonumber\\
\leq~&4\chi\sum\limits_{i=1}^3\int\partial_i^k\D_hw\cdot(\D\times \partial_i^ku) \dif x-\sum\limits_{i=1}^3\int\partial_i^k(u\cdot\D u)\cdot\partial_i^ku\dif x+\sum\limits_{i=1}^3\int\partial_i^k(B\cdot\D B)\cdot\partial_i^ku\dif x \nonumber\\
&-\sum\limits_{i=1}^3\int\partial_i^k(u\cdot\D B)\cdot\partial_i^kB\dif x+\sum\limits_{i=1}^3\int\partial_i^k(B\cdot\D u)\cdot\partial_i^kB\dif x-\sum\limits_{i=1}^3\int\partial_i^k(u\cdot\D w)\cdot\partial_i^kw\dif x \nonumber\\
\leq~&\sum\limits_{i=1}^3\left(\chi\|\D\partial_i^ku\|_{L^2}^2+4\chi\|\partial_i^kw\|_{L^2}^2\right)-\sum\limits_{i=1}^3\int\partial_i^k(u\cdot\D u)\cdot\partial_i^ku\dif x+\sum\limits_{i=1}^3\int\partial_i^k(B\cdot\D B)\cdot\partial_i^ku\dif x \nonumber\\
&-\sum\limits_{i=1}^3\int\partial_i^k(u\cdot\D B)\cdot\partial_i^kB\dif x+\sum\limits_{i=1}^3\int\partial_i^k(B\cdot\D u)\cdot\partial_i^kB\dif x-\sum\limits_{i=1}^3\int\partial_i^k(u\cdot\D w)\cdot\partial_i^kw\dif x \nonumber\\
:=~&\sum\limits_{i=1}^3\left(\chi\|\D\partial_i^ku\|_{L^2}^2+4\chi\|\partial_i^kw\|_{L^2}^2\right)+\sum\limits_{i=1}^5K_i.
\end{align}

Noting $\nabla\cdot u=0$, from Lemma \ref{pre.1} and Young inequality, we have
\begin{align}\label{2.2}
K_1=~&-\sum\limits_{i=1}^3\sum\limits_{j=1}^kC_k^j\int\partial_i^ju_h\cdot\D_h\partial_i^{k-j}u\cdot\partial_i^ku\dif x-\sum\limits_{i=1}^3\sum\limits_{j=1}^kC_k^j\int\partial_i^ju_3\partial_3\partial_i^{k-j}u\cdot\partial_i^ku\dif x \nonumber\\
\leq~&C\sum\limits_{i=1}^3\sum\limits_{j=1}^k\|\partial_i^ju\|_{L^2}^{\frac{1}{2}}\|\partial_1\partial_i^ju\|_{L^2}^{\frac{1}{2}}\|\D_h\partial_i^{k-j}u\|_{L^2}^{\frac{1}{2}}\|\partial_3\D_h\partial_i^{k-j}u\|_{L^2}^{\frac{1}{2}}\|\partial_i^ku\|_{L^2}^{\frac{1}{2}}\|\partial_2\partial_i^ku\|_{L^2}^{\frac{1}{2}} \nonumber\\
&+C\sum\limits_{i=1}^3\sum\limits_{j=1}^k\|\partial_i^ju_3\|_{L^2}^{\frac{1}{2}}\|\partial_3\partial_i^ju_3\|_{L^2}^{\frac{1}{2}}\|\partial_3\partial_i^{k-j}u\|_{L^2}^{\frac{1}{2}}\|\partial_2\partial_3\partial_i^{k-j}u\|_{L^2}^{\frac{1}{2}}\|\partial_i^ku\|_{L^2}^{\frac{1}{2}}\|\partial_1\partial_i^ku\|_{L^2}^{\frac{1}{2}} \nonumber\\
\leq~&C\|u\|_{H^k}\|\D_hu\|_{H^k}\|\D_hu\|_{H^{k-1}} \nonumber\\
&+C\sum\limits_{i=1}^3\sum\limits_{j=1}^k\|\partial_i^ju_3\|_{L^2}^{\frac{1}{2}}\|\partial_i^j\D_h\cdot u_h\|_{L^2}^{\frac{1}{2}}\|\partial_3\partial_i^{k-j}u\|_{L^2}^{\frac{1}{2}}\|\partial_2\partial_3\partial_i^{k-j}u\|_{L^2}^{\frac{1}{2}}\|\partial_i^ku\|_{L^2}^{\frac{1}{2}}\|\partial_1\partial_i^ku\|_{L^2}^{\frac{1}{2}} \nonumber\\
\leq~&C\|u\|_{H^k}\|\D_hu\|_{H^k}\|\D_hu\|_{H^{k-1}} \nonumber\\
\leq~&\frac{\mu}{5}\|\D_hu\|_{H^k}^2+C\|\D_hu\|_{H^{k-1}}^2\|u\|_{H^k}^2,
\end{align}
where $C_k^j$ are some positive constants that only depend of $j$ and $k$. Similarly, we have
\begin{align}\label{2.3}
K_3=~&-\sum\limits_{i=1}^3\sum\limits_{j=1}^kC_k^j\int\partial_i^ju_h\cdot\D_h\partial_i^{k-j}B\cdot\partial_i^kB\dif x-\sum\limits_{i=1}^3\sum\limits_{j=1}^kC_k^j\int\partial_i^ju_3\partial_3\partial_i^{k-j}B\cdot\partial_i^kB\dif x \nonumber\\
\leq~&C\sum\limits_{i=1}^3\sum\limits_{j=1}^k\|\partial_i^ju\|_{L^2}^{\frac{1}{2}}\|\partial_1\partial_i^ju\|_{L^2}^{\frac{1}{2}}\|\D_h\partial_i^{k-j}B\|_{L^2}^{\frac{1}{2}}\|\partial_3\D_h\partial_i^{k-j}B\|_{L^2}^{\frac{1}{2}}\|\partial_i^kB\|_{L^2}^{\frac{1}{2}}\|\partial_2\partial_i^kB\|_{L^2}^{\frac{1}{2}} \nonumber\\
&+C\sum\limits_{i=1}^3\sum\limits_{j=1}^k\|\partial_i^ju_3\|_{L^2}^{\frac{1}{2}}\|\partial_i^j\D_h\cdot u_h\|_{L^2}^{\frac{1}{2}}\|\partial_3\partial_i^{k-j}B\|_{L^2}^{\frac{1}{2}}\|\partial_2\partial_3\partial_i^{k-j}B\|_{L^2}^{\frac{1}{2}}\|\partial_i^kB\|_{L^2}^{\frac{1}{2}}\|\partial_1\partial_i^kB\|_{L^2}^{\frac{1}{2}} \nonumber\\
\leq~&\frac{1}{5}\left(\mu\|\D_hu\|_{H^k}^2+\nu\|\D_hB\|_{H^k}^2\right)+C\left(\|\D_hu\|_{H^{k-1}}^2+\|\D_hB\|_{H^{k-1}}^2\right)\left(\|u\|_{H^k}^2+\|B\|_{H^k}^2\right),
\end{align}
and
\begin{equation}\label{2.4}
K_5\leq\frac{1}{5}\left(\mu\|\D_hu\|_{H^k}^2+\gamma\|\D_hw\|_{H^k}^2\right)+C\left(\|\D_hu\|_{H^{k-1}}^2+\|\D_hw\|_{H^{k-1}}^2\right)\left(\|u\|_{H^k}^2+\|w\|_{H^k}^2\right).
\end{equation}

Noting $\nabla\cdot B=0$, we obtain after integration by parts that
\begin{align}\label{2.5}
K_2+K_4=~&\sum\limits_{i=1}^3\sum\limits_{j=1}^kC_k^j\int\partial_i^jB\cdot\D\partial_i^{k-j}B\cdot\partial_i^ku\dif x+\sum\limits_{i=1}^3\sum\limits_{j=1}^kC_k^j\int\partial_i^jB\cdot\D\partial_i^{k-j}u\cdot\partial_i^kB\dif x \nonumber\\
&+\int B\cdot\nabla\left(\partial_i^kB\cdot\partial_i^ku\right)\dif x \nonumber\\
=~&\sum\limits_{i=1}^3\sum\limits_{j=1}^kC_k^j\int\partial_i^jB_h\cdot\D_h\partial_i^{k-j}B\cdot\partial_i^ku\dif x+\sum\limits_{i=1}^3\sum\limits_{j=1}^kC_k^j\int\partial_i^jB_3\partial_3\partial_i^{k-j}B\cdot\partial_i^ku\dif x \nonumber\\
&+\sum\limits_{i=1}^3\sum\limits_{j=1}^kC_k^j\int\partial_i^jB_h\cdot\D_h\partial_i^{k-j}u\cdot\partial_i^kB\dif x+\sum\limits_{i=1}^3\sum\limits_{j=1}^kC_k^j\int\partial_i^jB_3\partial_3\partial_i^{k-j}u\cdot\partial_i^kB\dif x.
\end{align}

By Lemma \ref{pre.1}, Young inequality and the equation $\nabla\cdot B=0$, we have
\begin{align}\label{2.6}
K_2+K_4\leq~&C\|\partial_i^jB_h\|_{L^2}^{\frac{1}{2}}\|\partial_1\partial_i^jB_h\|_{L^2}^{\frac{1}{2}}\|\D_h\partial_i^{k-j}B\|_{L^2}^{\frac{1}{2}}\|\partial_3\D_h\partial_i^{k-j}B\|_{L^2}^{\frac{1}{2}}\|\partial_i^ku\|_{L^2}^{\frac{1}{2}}\|\partial_2\partial_i^ku\|_{L^2}^{\frac{1}{2}} \nonumber\\
&+C\|\partial_i^jB_3\|_{L^2}^{\frac{1}{2}}\|\partial_3\partial_i^jB_3\|_{L^2}^{\frac{1}{2}}\|\partial_3\partial_i^{k-j}B\|_{L^2}^{\frac{1}{2}}\|\partial_2\partial_3\partial_i^{k-j}B\|_{L^2}^{\frac{1}{2}}\|\partial_i^ku\|_{L^2}^{\frac{1}{2}}\|\partial_1\partial_i^ku\|_{L^2}^{\frac{1}{2}} \nonumber\\
&+C\|\partial_i^jB_h\|_{L^2}^{\frac{1}{2}}\|\partial_1\partial_i^jB_h\|_{L^2}^{\frac{1}{2}}\|\D_h\partial_i^{k-j}u\|_{L^2}^{\frac{1}{2}}\|\partial_3\D_h\partial_i^{k-j}u\|_{L^2}^{\frac{1}{2}}\|\partial_i^kB\|_{L^2}^{\frac{1}{2}}\|\partial_2\partial_i^kB\|_{L^2}^{\frac{1}{2}} \nonumber\\
&+C\|\partial_i^jB_3\|_{L^2}^{\frac{1}{2}}\|\partial_3\partial_i^jB_3\|_{L^2}^{\frac{1}{2}}\|\partial_3\partial_i^{k-j}u\|_{L^2}^{\frac{1}{2}}\|\partial_2\partial_3\partial_i^{k-j}u\|_{L^2}^{\frac{1}{2}}\|\partial_i^kB\|_{L^2}^{\frac{1}{2}}\|\partial_1\partial_i^kB\|_{L^2}^{\frac{1}{2}} \nonumber\\
\leq~&C\|(u, B)\|_{H^k}\|(\D_hu, \D_hB)\|_{H^{k-1}}\|(\D_hu, \D_hB)\|_{H^k} \nonumber\\
&+C\|\partial_i^jB_3\|_{L^2}^{\frac{1}{2}}\|\partial_i^j\D_h\cdot B_h\|_{L^2}^{\frac{1}{2}}\|\partial_3\partial_i^{k-j}B\|_{L^2}^{\frac{1}{2}}\|\partial_2\partial_3\partial_i^{k-j}B\|_{L^2}^{\frac{1}{2}}\|\partial_i^ku\|_{L^2}^{\frac{1}{2}}\|\partial_1\partial_i^ku\|_{L^2}^{\frac{1}{2}} \nonumber\\
&+C\|\partial_i^jB_3\|_{L^2}^{\frac{1}{2}}\|\partial_i^j\D_h\cdot B_h\|_{L^2}^{\frac{1}{2}}\|\partial_3\partial_i^{k-j}u\|_{L^2}^{\frac{1}{2}}\|\partial_2\partial_3\partial_i^{k-j}u\|_{L^2}^{\frac{1}{2}}\|\partial_i^kB\|_{L^2}^{\frac{1}{2}}\|\partial_1\partial_i^kB\|_{L^2}^{\frac{1}{2}} \nonumber\\
\leq~&C\|(u, B)\|_{H^k}\|(\D_hu, \D_hB)\|_{H^{k-1}}\|(\D_hu, \D_hB)\|_{H^k} \nonumber\\
\leq~&\frac{1}{5}\left(\mu\|\D_hu\|_{H^k}^2+\nu\|\D_hB\|_{H^k}^2\right)+C\left(\|\D_hu\|_{H^{k-1}}^2+\|\D_hB\|_{H^{k-1}}^2\right)\left(\|u\|_{H^k}^2+\|B\|_{H^k}^2\right).
\end{align}

Inserting \eqref{2.2}--\eqref{2.4} and \eqref{2.6} into \eqref{2.1}, then adding the resulting inequality to \eqref{1.1}, we obtain
\begin{equation}\label{2.7}
\frac{\dif}{\dif t}\|(u, B, w)\|_{H^k}^2+C_0\|(\D_hu, \D_hB, \D_hw)\|_{H^k}^2\leq C\|(\D_hu, \D_hB, \D_hw)\|_{H^{k-1}}^2\|(u, B, w)\|_{H^k}^2,
\end{equation}
for some positive constant $C_0$. By Gronwall inequality, we have proved \eqref{2.0}. This completes the proof of Theorem \ref{th}.
\end{proof}

\section{Proof of Theorem \ref{th2}}\label{Section.4}

In this section, we will establish the decay estimate of the solution under the following assumption.
\begin{equation}\label{3.0}
E_0:=\big\|\big(\Lambda_h^{-\sigma}u_0, \Lambda_h^{-\sigma}B_0, \Lambda_h^{-\sigma}w_0, \Lambda_h^{-\sigma}\partial_3u_0, \Lambda_h^{-\sigma}\partial_3B_0, \Lambda_h^{-\sigma}\partial_3w_0\big)\big\|_{L^2}^2<\infty.
\end{equation}

By the local well-posedness theory, there exists a finite time $T>0$, such that for any $t\in[0,T]$,
\begin{equation}\label{3.1}
\big\|\big(\Lambda_h^{-\sigma}u, \Lambda_h^{-\sigma}B, \Lambda_h^{-\sigma}w, \Lambda_h^{-\sigma}\partial_3u, \Lambda_h^{-\sigma}\partial_3B, \Lambda_h^{-\sigma}\partial_3w\big)(t)\big\|_{L^2}^2\leq 2E_0.
\end{equation}
If we can prove an improved estimate holds for any $t\in[0,T]$, then by the standard bootstrap argument, one finds that \eqref{3.1} holds for any $t\geq 0$. Thus our goal of this section is to prove the decay rate of the solution under the assumptions \eqref{3.0} and \eqref{3.1}, then improve \eqref{3.1} to close the estimates.

In the following lemma, we shall prove the $H^1$-decay estimates of the solution.
\begin{lemma}\label{lem.4}
Suppose the conditions of Theorem \ref{th} and Theorem \ref{th2} hold. Let $(u, B, w)$ be the global strong solution obtained in Theorem \ref{th} satisfying \eqref{3.1}. Then we have
\begin{equation}\label{4.0}
\|(u, B, w)(t)\|_{H^1}^2\leq C(1+t)^{-\sigma}.
\end{equation}
\end{lemma}
\begin{proof}
Taking $\varepsilon$ in \eqref{th.1} sufficiently small, we conclude from \eqref{1.18} that
\begin{equation}\label{4.1}
\frac{\dif}{\dif t}\left\|\left(u, B, w\right)\right\|_{H^1}^2+C_0\left\|\left(\D_hu, \D_hB, \D_hw\right)\right\|_{H^1}^2\leq 0.
\end{equation}

By \eqref{3.1}, Gagliardo-Nirenberg inequality and H\"older inequality, we have
\begin{align}\label{4.2}
\|u\|_{H^1}\leq~&C\left(\|u\|_{L^2}+\|\D_hu\|_{L^2}+\|\partial_3u\|_{L^2}\right) \nonumber\\
=~&C\left(\big\|\|u\|_{L^2_h}\big\|_{L^2_{x_3}}+\big\|\|\D_hu\|_{L^2_h}\big\|_{L^2_{x_3}}+\big\|\|\partial_3u\|_{L^2_h}\big\|_{L^2_{x_3}}\right) \nonumber\\
\leq~&C\left(\Big\|\|\Lambda_h^{-\sigma}u\|_{L^2_h}^{\frac{1}{1+\sigma}}\|\D_hu\|_{L^2_h}^{\frac{\sigma}{1+\sigma}}\Big\|_{L^2_{x_3}}+\Big\|\|\Lambda_h^{-\sigma}u\|_{L^2_h}^{\frac{1}{2+\sigma}}\|\Delta_hu\|_{L^2_h}^{\frac{1+\sigma}{2+\sigma}}\Big\|_{L^2_{x_3}}\right) \nonumber\\
&+C\Big\|\|\Lambda_h^{-\sigma}\partial_3u\|_{L^2_h}^{\frac{1}{1+\sigma}}\|\D_h\partial_3u\|_{L^2_h}^{\frac{\sigma}{1+\sigma}}\Big\|_{L^2_{x_3}} \nonumber\\
\leq~&C\left(\|\Lambda_h^{-\sigma}u\|_{L^2}^{\frac{1}{1+\sigma}}\|\D_hu\|_{L^2}^{\frac{\sigma}{1+\sigma}}+\|\Lambda_h^{-\sigma}u\|_{L^2}^{\frac{1}{2+\sigma}}\|\Delta_hu\|_{L^2}^{\frac{1+\sigma}{2+\sigma}}+\|\Lambda_h^{-\sigma}\partial_3u\|_{L^2}^{\frac{1}{1+\sigma}}\|\D_h\partial_3u\|_{L^2}^{\frac{\sigma}{1+\sigma}}\right) \nonumber\\
\leq~&C\|\D_hu\|_{H^1}^{\frac{\sigma}{1+\sigma}}.
\end{align}

Similarly, we have
\begin{equation}\label{4.3}
\|(B, w)\|_{H^1}\leq C\big\|\big(\D_hB, \D_hw\big)\big\|_{H^1}^{\frac{\sigma}{1+\sigma}}.
\end{equation}

Inserting \eqref{4.2} and \eqref{4.3} into \eqref{4.1}, we get
\begin{equation}\label{4.4}
\frac{\dif}{\dif t}\left\|\left(u, B, w\right)\right\|_{H^1}^2+C\left(\big\|\big(u, B, w\big)\big\|_{H^1}^2\right)^{\frac{1+\sigma}{\sigma}}\leq 0,
\end{equation}
based on which we are able to derive the desired estimate \eqref{4.0}.
\end{proof}

The next lemma is devoted to establishing an improved estimate of $\|(\D_hu, \D_hB, \D_hw)\|_{L^2}$.
\begin{lemma}\label{lem.5}
Suppose the conditions of Theorem \ref{th} and Theorem \ref{th2} hold. Let $(u, B, w)$ be the global strong solution obtained in Theorem \ref{th} satisfying \eqref{3.1}. Then we have
\begin{equation}\label{5.0}
\|(\D_hu, \D_hB, \D_hw)(t)\|_{L^2}^2\leq C(1+t)^{-(1+\sigma)}.
\end{equation}
\end{lemma}
\begin{proof}
Recalling \eqref{1.9}--\eqref{1.15}, we have
\begin{align}\label{5.1}
&\frac{1}{2}\frac{\dif}{\dif t}\left\|\left(\D_hu, \D_hB, \D_hw\right)\right\|_{L^2}^2+\mu\|\Delta_hu\|_{L^2}^2+\nu\|\Delta_hB\|_{L^2}^2+\gamma\|\Delta_hw\|_{L^2}^2 \nonumber\\
\leq~&\int u\cdot\D u\cdot\Delta_hu\dif x-\int B\cdot\D B\cdot\Delta_hu\dif x+\int u\cdot\D B\cdot\Delta_hB\dif x \nonumber\\
&-\int B\cdot\D u\cdot\Delta_hB\dif x+\int u\cdot\D w\cdot\Delta_hw\dif x \nonumber\\
\leq~&C\|(\D_hu, \D_hB, \D_hw)\|_{L^3}^3+C\int|\D_hu_3|\Big(|\D_hu||\partial_3u|+|\D_hB||\partial_3B|+|\D_hw||\partial_3w|\Big)\dif x \nonumber\\
&+C\int|\D_hB_3|\big(|\D_hu|+|\D_hB|\big)\big(|\partial_3u|+|\partial_3B|\big)\dif x.
\end{align}

By Lemma \ref{pre.1} and Young inequality, we have
\begin{align}\label{5.2}
&C\|(\D_hu, \D_hB, \D_hw)\|_{L^3}^3 \nonumber\\
\leq~&C\|\D_hu\|_{L^2}^{\frac{3}{2}}\|\partial_1\D_hu\|_{L^2}^{\frac{1}{2}}\|\partial_2\D_hu\|_{L^2}^{\frac{1}{2}}\|\partial_3\D_hu\|_{L^2}^{\frac{1}{2}} \nonumber\\
&+C\|\D_hB\|_{L^2}^{\frac{3}{2}}\|\partial_1\D_hB\|_{L^2}^{\frac{1}{2}}\|\partial_2\D_hB\|_{L^2}^{\frac{1}{2}}\|\partial_3\D_hB\|_{L^2}^{\frac{1}{2}} \nonumber\\
&+C\|\D_hw\|_{L^2}^{\frac{3}{2}}\|\partial_1\D_hw\|_{L^2}^{\frac{1}{2}}\|\partial_2\D_hw\|_{L^2}^{\frac{1}{2}}\|\partial_3\D_hw\|_{L^2}^{\frac{1}{2}} \nonumber\\
\leq~&C\|(\D_hu, \D_hB, \D_hw)\|_{L^2}^{\frac{3}{2}}\|(\Delta_hu, \Delta_hB, \Delta_hw)\|_{L^2}\|(\D_h\partial_3u, \D_h\partial_3B, \D_h\partial_3w)\|_{L^2}^{\frac{1}{2}} \nonumber\\
\leq~&\frac{1}{10}\left(\mu\|\Delta_hu\|_{L^2}^2+\nu\|\Delta_hB\|_{L^2}^2+\gamma\|\Delta_hw\|_{L^2}^2\right) \nonumber\\
&+C\|(\D_hu, \D_hB, \D_hw)\|_{L^2}^2\|(\D_hu, \D_hB, \D_hw, \D_h\partial_3u, \D_h\partial_3B, \D_h\partial_3w)\|_{L^2}^2,
\end{align}
\begin{align}\label{5.3}
&C\int|\D_hu_3|\Big(|\D_hu||\partial_3u|+|\D_hB||\partial_3B|+|\D_hw||\partial_3w|\Big)\dif x \nonumber\\
\leq~&C\|\D_hu_3\|_{L^2}^{\frac{1}{2}}\|\D_h\partial_3u_3\|_{L^2}^{\frac{1}{2}}\|\D_hu\|_{L^2}^{\frac{1}{2}}\|\partial_2\D_hu\|_{L^2}^{\frac{1}{2}}\|\partial_3u\|_{L^2}^{\frac{1}{2}}\|\partial_1\partial_3u\|_{L^2}^{\frac{1}{2}} \nonumber\\
&+C\|\D_hu_3\|_{L^2}^{\frac{1}{2}}\|\D_h\partial_3u_3\|_{L^2}^{\frac{1}{2}}\|\D_hB\|_{L^2}^{\frac{1}{2}}\|\partial_2\D_hB\|_{L^2}^{\frac{1}{2}}\|\partial_3B\|_{L^2}^{\frac{1}{2}}\|\partial_1\partial_3B\|_{L^2}^{\frac{1}{2}} \nonumber\\
&+C\|\D_hu_3\|_{L^2}^{\frac{1}{2}}\|\D_h\partial_3u_3\|_{L^2}^{\frac{1}{2}}\|\D_hw\|_{L^2}^{\frac{1}{2}}\|\partial_2\D_hw\|_{L^2}^{\frac{1}{2}}\|\partial_3w\|_{L^2}^{\frac{1}{2}}\|\partial_1\partial_3w\|_{L^2}^{\frac{1}{2}} \nonumber\\
\leq~& C\|\D_hu_3\|_{L^2}^{\frac{1}{2}}\|\D_h\D_h\cdot u_h\|_{L^2}^{\frac{1}{2}}\|\D_hu\|_{L^2}^{\frac{1}{2}}\|\partial_2\D_hu\|_{L^2}^{\frac{1}{2}}\|\partial_3u\|_{L^2}^{\frac{1}{2}}\|\partial_1\partial_3u\|_{L^2}^{\frac{1}{2}}\nonumber\\
&+C\|\D_hu_3\|_{L^2}^{\frac{1}{2}}\|\D_h\D_h\cdot u_h\|_{L^2}^{\frac{1}{2}}\|\D_hB\|_{L^2}^{\frac{1}{2}}\|\partial_2\D_hB\|_{L^2}^{\frac{1}{2}}\|\partial_3B\|_{L^2}^{\frac{1}{2}}\|\partial_1\partial_3B\|_{L^2}^{\frac{1}{2}} \nonumber\\
&+C\|\D_hu_3\|_{L^2}^{\frac{1}{2}}\|\D_h\D_h\cdot u_h\|_{L^2}^{\frac{1}{2}}\|\D_hw\|_{L^2}^{\frac{1}{2}}\|\partial_2\D_hw\|_{L^2}^{\frac{1}{2}}\|\partial_3w\|_{L^2}^{\frac{1}{2}}\|\partial_1\partial_3w\|_{L^2}^{\frac{1}{2}} \nonumber\\
\leq~&C\|(\D_hu, \D_hB, \D_hw)\|_{L^2}\|(\partial_3u, \partial_3B, \partial_3w)\|_{L^2}^{\frac{1}{2}}\|(\Delta_hu, \Delta_hB, \Delta_hw)\|_{L^2}\|(\D_h\partial_3u, \D_h\partial_3B, \D_h\partial_3w)\|_{L^2}^{\frac{1}{2}} \nonumber\\
\leq~&\frac{1}{10}\left(\mu\|\Delta_hu\|_{L^2}^2+\nu\|\Delta_hB\|_{L^2}^2+\gamma\|\Delta_hw\|_{L^2}^2\right) \nonumber\\
&+C\|(\D_hu, \D_hB, \D_hw)\|_{L^2}\|(\partial_3u, \partial_3B, \partial_3w)\|_{L^2}\|(\D_hu, \D_hB, \D_hw, \D_h\partial_3u, \D_h\partial_3B, \D_h\partial_3w)\|_{L^2}^2,
\end{align}
and
\begin{align}\label{5.4}
&C\int|\D_hB_3|\big(|\D_hu|+|\D_hB|\big)\big(|\partial_3u|+|\partial_3B|\big)\dif x \nonumber\\
\leq~&C\|\D_hB_3\|_{L^2}^{\frac{1}{2}}\|\D_h\partial_3B_3\|_{L^2}^{\frac{1}{2}}\|(\D_hu, \D_hB)\|_{L^2}^{\frac{1}{2}}\|(\partial_2\D_hu, \partial_2\D_hB)\|_{L^2}^{\frac{1}{2}}\|(\partial_3u, \partial_3B)\|_{L^2}^{\frac{1}{2}}\|(\partial_1\partial_3u, \partial_1\partial_3B)\|_{L^2}^{\frac{1}{2}} \nonumber\\
\leq~&C\|\D_hB_3\|_{L^2}^{\frac{1}{2}}\|\D_h\D_h\cdot B_h\|_{L^2}^{\frac{1}{2}}\|(\D_hu, \D_hB)\|_{L^2}^{\frac{1}{2}}\|(\partial_2\D_hu, \partial_2\D_hB)\|_{L^2}^{\frac{1}{2}}\|(\partial_3u, \partial_3B)\|_{L^2}^{\frac{1}{2}}\|(\partial_1\partial_3u, \partial_1\partial_3B)\|_{L^2}^{\frac{1}{2}} \nonumber\\
\leq~&C\|(\D_hu, \D_hB)\|_{L^2}\|(\partial_3u, \partial_3B)\|_{L^2}^{\frac{1}{2}}\|(\Delta_hu, \Delta_hB)\|_{L^2}\|(\D_h\partial_3u, \D_h\partial_3B)\|_{L^2}^{\frac{1}{2}} \nonumber\\
\leq~&\frac{1}{10}\left(\mu\|\Delta_hu\|_{L^2}^2+\nu\|\Delta_hB\|_{L^2}^2\right) \nonumber\\
&+C\|(\D_hu, \D_hB)\|_{L^2}\|(\partial_3u, \partial_3B)\|_{L^2}\|(\D_hu, \D_hB, \D_h\partial_3u, \D_h\partial_3B)\|_{L^2}^2.
\end{align}

Substituting \eqref{5.2}--\eqref{5.4} into \eqref{5.1} and noting \eqref{4.0}, we obtain
\begin{align}\label{5.5}
&\frac{\dif}{\dif t}\left\|\left(\D_hu, \D_hB, \D_hw\right)(t)\right\|_{L^2}^2+C_0\|(\Delta_hu, \Delta_hB, \Delta_hw)(t)\|_{L^2}^2 \nonumber\\
\leq~&C\|(\D_hu, \D_hB, \D_hw)\|_{L^2}^2\|(\D_hu, \D_hB, \D_hw, \D_h\partial_3u, \D_h\partial_3B, \D_h\partial_3w)(t)\|_{L^2}^2 \nonumber\\
&C\|(\D_hu, \D_hB, \D_hw)\|_{L^2}\|(\partial_3u, \partial_3B, \partial_3w)\|_{L^2}\|(\D_hu, \D_hB, \D_hw, \D_h\partial_3u, \D_h\partial_3B, \D_h\partial_3w)(t)\|_{L^2}^2 \nonumber\\
\leq~&C(1+t)^{-\sigma}\|(\D_hu, \D_hB, \D_hw, \D_h\partial_3u, \D_h\partial_3B, \D_h\partial_3w)(t)\|_{L^2}^2.
\end{align}

Multiplying \eqref{5.5} by $t-s$ for fixed $s\in(0,t)$, we get
\begin{align*}
&\frac{\dif}{\dif t}\left((t-s)\left\|\left(\D_hu, \D_hB, \D_hw\right)(t)\right\|_{L^2}^2\right)+C_0(t-s)\|(\Delta_hu, \Delta_hB, \Delta_hw)(t)\|_{L^2}^2 \\
\leq~&C\left\|\left(\D_hu, \D_hB, \D_hw\right)(t)\right\|_{L^2}^2+C(t-s)(1+t)^{-\sigma}\|(\D_hu, \D_hB, \D_hw, \D_h\partial_3u, \D_h\partial_3B, \D_h\partial_3w)(t)\|_{L^2}^2.
\end{align*}
Integrating the above inequality over $(s,t)$, we obtain
\begin{align}\label{5.6}
&(t-s)\left\|\left(\D_hu, \D_hB, \D_hw\right)(t)\right\|_{L^2}^2+C_0\int_s^t(\tau-s)\|(\Delta_hu, \Delta_hB, \Delta_hw)(\tau)\|_{L^2}^2\dif\tau \nonumber\\
\leq~&C\int_s^t\left\|\left(\D_hu, \D_hB, \D_hw\right)(\tau)\right\|_{L^2}^2\dif\tau \nonumber\\
&+C\int_s^t(\tau-s)(1+\tau)^{-\sigma}\|(\D_hu, \D_hB, \D_hw, \D_h\partial_3u, \D_h\partial_3B, \D_h\partial_3w)(\tau)\|_{L^2}^2\dif\tau.
\end{align}

Taking $s=\frac{t}{2}$ in \eqref{5.6} and noting \eqref{1.8}, we have
\begin{align}\label{5.7}
&\frac{t}{2}\left\|\left(\D_hu, \D_hB, \D_hw\right)(t)\right\|_{L^2}^2+C_0\int_{\frac{t}{2}}^t\Big(\tau-\frac{t}{2}\Big)\|(\Delta_hu, \Delta_hB, \Delta_hw)(\tau)\|_{L^2}^2\dif\tau \nonumber\\
\leq~&C\int_{\frac{t}{2}}^t\left\|\left(\D_hu, \D_hB, \D_hw\right)(\tau)\right\|_{L^2}^2\dif\tau \nonumber\\
&+C\int_{\frac{t}{2}}^t\Big(\tau-\frac{t}{2}\Big)(1+\tau)^{-\sigma}\|(\D_hu, \D_hB, \D_hw, \D_h\partial_3u, \D_h\partial_3B, \D_h\partial_3w)(\tau)\|_{L^2}^2\dif\tau \nonumber\\
\leq~&C\int_{\frac{t}{2}}^t\left\|\left(\D_hu, \D_hB, \D_hw\right)(\tau)\right\|_{L^2}^2\dif\tau \nonumber\\
&+C(1+t)^{1-\sigma}\int_{\frac{t}{2}}^t\|(\D_hu, \D_hB, \D_hw, \D_h\partial_3u, \D_h\partial_3B, \D_h\partial_3w)(\tau)\|_{L^2}^2\dif\tau \nonumber\\
\leq~&C\left(1+(1+t)^{1-\sigma}\right)\big\|(u, B, w, \partial_3u, \partial_3B, \partial_3w)\Big(\frac{t}{2}\Big)\big\|_{L^2}^2 \nonumber\\
\leq~&C(1+t)^{1-2\sigma},
\end{align}
which implies
\begin{equation}\label{5.8}
\left\|\left(\D_hu, \D_hB, \D_hw\right)\right\|_{L^2}^2\leq C(1+t)^{-2\sigma}.
\end{equation}

Motivated by \cite{MR4850530}, we shall implement an iterative procedure. Denote $a_0=2\sigma$. Inserting \eqref{4.0} and \eqref{5.8} into \eqref{5.5}, we obtain
\begin{align}\label{5.9}
&\frac{\dif}{\dif t}\left\|\left(\D_hu, \D_hB, \D_hw\right)(t)\right\|_{L^2}^2+C_0\|(\Delta_hu, \Delta_hB, \Delta_hw)(t)\|_{L^2}^2 \nonumber\\
\leq~&C(1+t)^{-\frac{3}{2}\sigma}\|(\D_hu, \D_hB, \D_hw, \D_h\partial_3u, \D_h\partial_3B, \D_h\partial_3w)(t)\|_{L^2}^2.
\end{align}

Similar to \eqref{5.6}--\eqref{5.7}, we have
\begin{equation}\label{5.10}
\left\|\left(\D_hu, \D_hB, \D_hw\right)\right\|_{L^2}^2\leq C(1+t)^{-(1+\sigma)}+C(1+t)^{-a_1},
\end{equation}
with $a_1=\frac{1}{2}(a_0-\sigma)+a_0$. Repeating the process for $n$ times, we get
\begin{equation}\label{5.11}
\left\|\left(\D_hu, \D_hB, \D_hw\right)\right\|_{L^2}^2\leq C(1+t)^{-(1+\sigma)}+C(1+t)^{-a_n},
\end{equation}
with
$$a_n=\frac{1}{2}(a_{n-1}-\sigma)+a_0.$$
Direct computation shows that
$$a_n=3\sigma-\frac{\sigma}{2^n}\to 3\sigma \quad \text{as}\quad n\to\infty.$$
Noting that $\frac{k-1}{2(k-2)}<\sigma<1$ for $k\geq 4$, we have $\sigma>\frac{1}{2}$ and thus $3\sigma<1+\sigma$. This implies \eqref{5.0}.
\end{proof}

At last, we will derive a uniform bound of $\|(\Lambda_h^{-\sigma}u, \Lambda_h^{-\sigma}B, \Lambda_h^{-\sigma}w, \Lambda_h^{-\sigma}\partial_3u, \Lambda_h^{-\sigma}\partial_3B, \Lambda_h^{-\sigma}\partial_3w)\|_{L^2}$.
\begin{lemma}
Suppose the conditions of Theorem \ref{th} and Theorem \ref{th2} hold. Let $(u, B, w)$ be the global strong solution obtained in Theorem \ref{th} satisfying \eqref{3.1}. Then the following estimate holds for any $t\in[0,T]$.
\begin{equation}\label{6.0}
\big\|\big(\Lambda_h^{-\sigma}u, \Lambda_h^{-\sigma}B, \Lambda_h^{-\sigma}w, \Lambda_h^{-\sigma}\partial_3u, \Lambda_h^{-\sigma}\partial_3B, \Lambda_h^{-\sigma}\partial_3w\big)(t)\big\|_{L^2}^2\leq \frac{3}{2}E_0.
\end{equation}
\end{lemma}
\begin{proof}
Applying $\Lambda_h^{-\sigma}$ to \eqref{BMHD}, and taking the $L^2$-inner product with $(\Lambda_h^{-\sigma}u, \Lambda_h^{-\sigma}B, \Lambda_h^{-\sigma}w)$, we obtain
\begin{align}\label{6.1}
&\frac{1}{2}\frac{\dif}{\dif t}\|(\Lambda_h^{-\sigma}u, \Lambda_h^{-\sigma}B, \Lambda_h^{-\sigma}w)\|_{L^2}^2+\mu\|\Lambda_h^{1-\sigma}u\|_{L^2}^2+\nu\|\Lambda_h^{1-\sigma}B\|_{L^2}^2+\gamma\|\Lambda_h^{1-\sigma}w\|_{L^2}^2 \nonumber\\
&+\chi\|\D\Lambda_h^{-\sigma}u\|_{L^2}^2+4\chi\|\Lambda_h^{-\sigma}w\|_{L^2}^2 \nonumber\\
\leq~&4\chi\int\Lambda_h^{-\sigma}w\cdot\Lambda_h^{-\sigma}(\D\times u) \dif x-\int\Lambda_h^{-\sigma}(u\cdot\D u)\cdot\Lambda_h^{-\sigma}u\dif x+\int\Lambda_h^{-\sigma}(B\cdot\D B)\cdot\Lambda_h^{-\sigma}u\dif x \nonumber\\
&-\int\Lambda_h^{-\sigma}(u\cdot\D B)\cdot\Lambda_h^{-\sigma}B\dif x+\int\Lambda_h^{-\sigma}(B\cdot\D u)\cdot\Lambda_h^{-\sigma}B\dif x-\int\Lambda_h^{-\sigma}(u\cdot\D w)\cdot\Lambda_h^{-\sigma}w\dif x \nonumber\\
\leq~&\chi\|\D\Lambda_h^{-\sigma}u\|_{L^2}^2+4\chi\|\Lambda_h^{-\sigma}w\|_{L^2}^2-\int\Lambda_h^{-\sigma}(u\cdot\D u)\cdot\Lambda_h^{-\sigma}u\dif x+\int\Lambda_h^{-\sigma}(B\cdot\D B)\cdot\Lambda_h^{-\sigma}u\dif x \nonumber\\
&-\int\Lambda_h^{-\sigma}(u\cdot\D B)\cdot\Lambda_h^{-\sigma}B\dif x+\int\Lambda_h^{-\sigma}(B\cdot\D u)\cdot\Lambda_h^{-\sigma}B\dif x-\int\Lambda_h^{-\sigma}(u\cdot\D w)\cdot\Lambda_h^{-\sigma}w\dif x \nonumber\\
:=~&\chi\|\D\Lambda_h^{-\sigma}u\|_{L^2}^2+4\chi\|\Lambda_h^{-\sigma}w\|_{L^2}^2+\sum\limits_{i=1}^5L_i.
\end{align}

By H\"older inequality, Sobolev inequality, Lemma \ref{pre.2} and the equation $\nabla\cdot u=0$, we have
\begin{align}\label{6.2}
L_1=~&-\int\Lambda_h^{-\sigma}(u_h\cdot\D_hu)\cdot\Lambda_h^{-\sigma}u\dif x-\int\Lambda_h^{-\sigma}(u_3\partial_3u)\cdot\Lambda_h^{-\sigma}u\dif x \nonumber\\
\leq~&C\|\Lambda_h^{-\sigma}u\|_{L^2}\Big(\big\|\Lambda_h^{-\sigma}(u_h\cdot\D_hu)\big\|_{L^2}+\big\|\Lambda_h^{-\sigma}(u_3\partial_3u)\big\|_{L^2}\Big) \nonumber\\
\leq~&C\|\Lambda_h^{-\sigma}u\|_{L^2}\Big(\Big\|\|u_h\|_{L_h^{\frac{2}{\sigma}}}\|\D_hu\|_{L_h^2}\Big\|_{L^2_{x_3}}+\big\|\|u_3\|_{L_h^{\frac{2}{\sigma}}}\|\partial_3u\|_{L_h^2}\big\|_{L^2_{x_3}}\Big) \nonumber\\
\leq~&C\|\Lambda_h^{-\sigma}u\|_{L^2}\Big(\big\|\|u_h\|_{L^\infty_{x_3}}\big\|_{L_h^{\frac{2}{\sigma}}}\|\D_hu\|_{L^2}+\big\|\|u_3\|_{L^\infty_{x_3}}\big\|_{L_h^{\frac{2}{\sigma}}}\|\partial_3u\|_{L^2}\Big) \nonumber\\
\leq~&C\|\Lambda_h^{-\sigma}u\|_{L^2}\|\D_hu\|_{L^2}\Big\|\|u_h\|_{L^2_{x_3}}^{\frac{1}{2}}\|\partial_3u_h\|_{L^2_{x_3}}^{\frac{1}{2}}\Big\|_{L_h^{\frac{2}{\sigma}}} \nonumber\\
&+C\|\Lambda_h^{-\sigma}u\|_{L^2}\|\partial_3u\|_{L^2}\Big\|\|u_3\|_{L^2_{x_3}}^{\frac{1}{2}}\|\partial_3u_3\|_{L^2_{x_3}}^{\frac{1}{2}}\Big\|_{L_h^{\frac{2}{\sigma}}} \nonumber\\
\leq~&C\|\Lambda_h^{-\sigma}u\|_{L^2}\|\D_hu\|_{L^2}\Big\|\|u_h\|_{L_h^2}^{2\sigma-1}\|\D_hu_h\|_{L_h^2}^{2-2\sigma}\Big\|_{L^2_{x_3}}^{\frac{1}{2}}\|\partial_3u_h\|_{L^2}^{\frac{1}{2}} \nonumber\\
&+C\|\Lambda_h^{-\sigma}u\|_{L^2}\|\partial_3u\|_{L^2}\Big\|\|u_3\|_{L_h^2}^{2\sigma-1}\|\D_hu_3\|_{L_h^2}^{2-2\sigma}\Big\|_{L^2_{x_3}}^{\frac{1}{2}}\|\partial_3u_3\|_{L^2}^{\frac{1}{2}} \nonumber\\
\leq~&C\|\Lambda_h^{-\sigma}u\|_{L^2}\|\D_hu\|_{L^2}\|u_h\|_{L^2}^{\sigma-\frac{1}{2}}\|\D_hu\|_{L^2}^{1-\sigma}\|\partial_3u_h\|_{L^2}^{\frac{1}{2}} \nonumber\\
&+C\|\Lambda_h^{-\sigma}u\|_{L^2}\|\partial_3u\|_{L^2}\|u_3\|_{L^2}^{\sigma-\frac{1}{2}}\|\D_hu\|_{L^2}^{1-\sigma}\|\partial_3u_3\|_{L^2}^{\frac{1}{2}}
 \nonumber\\
\leq~&C\|u\|_{L^2}^{\sigma-\frac{1}{2}}\|\Lambda_h^{-\sigma}u\|_{L^2}\Big(\|\D_hu\|_{L^2}^{2-\sigma}\|\partial_3u\|_{L^2}^{\frac{1}{2}}+\|\D_hu\|_{L^2}^{\frac{3}{2}-\sigma}\|\partial_3u\|_{L^2}\Big).
\end{align}

Similarly, we have
\begin{align}\label{6.3}
L_3+L_5\leq~&C\|\Lambda_h^{-\sigma}B\|_{L^2}\Big(\big\|\Lambda_h^{-\sigma}(u_h\cdot\D_hB)\big\|_{L^2}+\big\|\Lambda_h^{-\sigma}(u_3\partial_3B)\big\|_{L^2}\Big) \nonumber\\
&+C\|\Lambda_h^{-\sigma}w\|_{L^2}\Big(\big\|\Lambda_h^{-\sigma}(u_h\cdot\D_hw)\big\|_{L^2}+\big\|\Lambda_h^{-\sigma}(u_3\partial_3w)\big\|_{L^2}\Big) \nonumber\\
\leq~&C\|\Lambda_h^{-\sigma}B\|_{L^2}\Big(\Big\|\|u_h\|_{L_h^{\frac{2}{\sigma}}}\|\D_hB\|_{L_h^2}\Big\|_{L^2_{x_3}}+\big\|\|u_3\|_{L_h^{\frac{2}{\sigma}}}\|\partial_3B\|_{L_h^2}\big\|_{L^2_{x_3}}\Big) \nonumber\\
&+C\|\Lambda_h^{-\sigma}w\|_{L^2}\Big(\Big\|\|u_h\|_{L_h^{\frac{2}{\sigma}}}\|\D_hw\|_{L_h^2}\Big\|_{L^2_{x_3}}+\big\|\|u_3\|_{L_h^{\frac{2}{\sigma}}}\|\partial_3w\|_{L_h^2}\big\|_{L^2_{x_3}}\Big) \nonumber\\
\leq~&C\|\Lambda_h^{-\sigma}B\|_{L^2}\Big(\big\|\|u_h\|_{L^\infty_{x_3}}\big\|_{L_h^{\frac{2}{\sigma}}}\|\D_hB\|_{L^2}+\big\|\|u_3\|_{L^\infty_{x_3}}\big\|_{L_h^{\frac{2}{\sigma}}}\|\partial_3B\|_{L^2}\Big) \nonumber\\
&+C\|\Lambda_h^{-\sigma}w\|_{L^2}\Big(\big\|\|u_h\|_{L^\infty_{x_3}}\big\|_{L_h^{\frac{2}{\sigma}}}\|\D_hw\|_{L^2}+\big\|\|u_3\|_{L^\infty_{x_3}}\big\|_{L_h^{\frac{2}{\sigma}}}\|\partial_3w\|_{L^2}\Big) \nonumber\\
\leq~&C\|(\Lambda_h^{-\sigma}B, \Lambda_h^{-\sigma}w)\|_{L^2}\|(\D_hB, \D_hw)\|_{L^2}\|u_h\|_{L^2}^{\sigma-\frac{1}{2}}\|\D_hu\|_{L^2}^{1-\sigma}\|\partial_3u_h\|_{L^2}^{\frac{1}{2}} \nonumber\\
&+C\|(\Lambda_h^{-\sigma}B, \Lambda_h^{-\sigma}w)\|_{L^2}\|(\partial_3B, \partial_3w)\|_{L^2}\|u_3\|_{L^2}^{\sigma-\frac{1}{2}}\|\D_hu\|_{L^2}^{1-\sigma}\|\partial_3u_3\|_{L^2}^{\frac{1}{2}}
 \nonumber\\
\leq~&C\|u\|_{L^2}^{\sigma-\frac{1}{2}}\|(\Lambda_h^{-\sigma}B, \Lambda_h^{-\sigma}w)\|_{L^2}\|\D_hu\|_{L^2}^{1-\sigma}\|\partial_3u\|_{L^2}^{\frac{1}{2}}\|(\D_hB, \D_hw)\|_{L^2} \nonumber\\
&+C\|u\|_{L^2}^{\sigma-\frac{1}{2}}\|(\Lambda_h^{-\sigma}B, \Lambda_h^{-\sigma}w)\|_{L^2}\|\D_hu\|_{L^2}^{\frac{3}{2}-\sigma}\|(\partial_3B, \partial_3w)\|_{L^2},
\end{align}
and
\begin{align}\label{6.4}
L_2+L_4\leq~&C\|\Lambda_h^{-\sigma}u\|_{L^2}\Big(\big\|\Lambda_h^{-\sigma}(B_h\cdot\D_hB)\big\|_{L^2}+\big\|\Lambda_h^{-\sigma}(B_3\partial_3B)\big\|_{L^2}\Big) \nonumber\\
&+C\|\Lambda_h^{-\sigma}B\|_{L^2}\Big(\big\|\Lambda_h^{-\sigma}(B_h\cdot\D_hu)\big\|_{L^2}+\big\|\Lambda_h^{-\sigma}(B_3\partial_3u)\big\|_{L^2}\Big) \nonumber\\
\leq~&C\|\Lambda_h^{-\sigma}u\|_{L^2}\Big(\Big\|\|B_h\|_{L_h^{\frac{2}{\sigma}}}\|\D_hB\|_{L_h^2}\Big\|_{L^2_{x_3}}+\big\|\|B_3\|_{L_h^{\frac{2}{\sigma}}}\|\partial_3B\|_{L_h^2}\big\|_{L^2_{x_3}}\Big) \nonumber\\
&+C\|\Lambda_h^{-\sigma}B\|_{L^2}\Big(\Big\|\|B_h\|_{L_h^{\frac{2}{\sigma}}}\|\D_hu\|_{L_h^2}\Big\|_{L^2_{x_3}}+\big\|\|B_3\|_{L_h^{\frac{2}{\sigma}}}\|\partial_3u\|_{L_h^2}\big\|_{L^2_{x_3}}\Big) \nonumber\\
\leq~&C\|\Lambda_h^{-\sigma}u\|_{L^2}\Big(\big\|\|B_h\|_{L^\infty_{x_3}}\big\|_{L_h^{\frac{2}{\sigma}}}\|\D_hB\|_{L^2}+\big\|\|B_3\|_{L^\infty_{x_3}}\big\|_{L_h^{\frac{2}{\sigma}}}\|\partial_3B\|_{L^2}\Big) \nonumber\\
&+C\|\Lambda_h^{-\sigma}B\|_{L^2}\Big(\big\|\|B_h\|_{L^\infty_{x_3}}\big\|_{L_h^{\frac{2}{\sigma}}}\|\D_hu\|_{L^2}+\big\|\|B_3\|_{L^\infty_{x_3}}\big\|_{L_h^{\frac{2}{\sigma}}}\|\partial_3u\|_{L^2}\Big) \nonumber\\
\leq~&C\|(\Lambda_h^{-\sigma}u, \Lambda_h^{-\sigma}B)\|_{L^2}\|(\D_hu, \D_hB)\|_{L^2}\|B_h\|_{L^2}^{\sigma-\frac{1}{2}}\|\D_hB\|_{L^2}^{1-\sigma}\|\partial_3B_h\|_{L^2}^{\frac{1}{2}} \nonumber\\
&+C\|(\Lambda_h^{-\sigma}u, \Lambda_h^{-\sigma}B)\|_{L^2}\|(\partial_3u, \partial_3B)\|_{L^2}\|B_3\|_{L^2}^{\sigma-\frac{1}{2}}\|\D_hB\|_{L^2}^{1-\sigma}\|\partial_3B_3\|_{L^2}^{\frac{1}{2}}
 \nonumber\\
\leq~&C\|B\|_{L^2}^{\sigma-\frac{1}{2}}\|(\Lambda_h^{-\sigma}u, \Lambda_h^{-\sigma}B)\|_{L^2}\|\D_hB\|_{L^2}^{1-\sigma}\|\partial_3B\|_{L^2}^{\frac{1}{2}}\|(\D_hu, \D_hB)\|_{L^2} \nonumber\\
&+C\|B\|_{L^2}^{\sigma-\frac{1}{2}}\|(\Lambda_h^{-\sigma}u, \Lambda_h^{-\sigma}B)\|_{L^2}\|\D_hB\|_{L^2}^{\frac{3}{2}-\sigma}\|(\partial_3u, \partial_3B)\|_{L^2}.
\end{align}

Inserting \eqref{6.2}--\eqref{6.4} into \eqref{6.1}, we obtain
\begin{align}\label{6.5}
&\frac{1}{2}\frac{\dif}{\dif t}\|(\Lambda_h^{-\sigma}u, \Lambda_h^{-\sigma}B, \Lambda_h^{-\sigma}w)\|_{L^2}^2+\mu\|\Lambda_h^{1-\sigma}u\|_{L^2}^2+\nu\|\Lambda_h^{1-\sigma}B\|_{L^2}^2+\gamma\|\Lambda_h^{1-\sigma}w\|_{L^2}^2 \nonumber\\
\leq~&C\|(u, B)\|_{L^2}^{\sigma-\frac{1}{2}}\|(\Lambda_h^{-\sigma}u, \Lambda_h^{-\sigma}B, \Lambda_h^{-\sigma}w)\|_{L^2}\|(\D_hu, \D_hB)\|_{L^2}^{1-\sigma}\|(\partial_3u, \partial_3B)\|_{L^2}^{\frac{1}{2}}\|(\D_hu, \D_hB, \D_hw)\|_{L^2} \nonumber\\
&+C\|(u, B)\|_{L^2}^{\sigma-\frac{1}{2}}\|(\Lambda_h^{-\sigma}u, \Lambda_h^{-\sigma}B, \Lambda_h^{-\sigma}w)\|_{L^2}\|(\D_hu, \D_hB)\|_{L^2}^{\frac{3}{2}-\sigma}\|(\partial_3u, \partial_3B, \partial_3w)\|_{L^2}.
\end{align}
For simplicity, we denote
\begin{equation}\label{UV}
U=(u, B, w), \quad V=(u, B).
\end{equation}
Then \eqref{6.5} can be rewritten as
\begin{align}\label{6.5'}
&\frac{\dif}{\dif t}\|\Lambda_h^{-\sigma}U\|_{L^2}^2+C_0\|\Lambda_h^{1-\sigma}U\|_{L^2}^2 \nonumber\\
\leq~&C\|V\|_{L^2}^{\sigma-\frac{1}{2}}\|\Lambda_h^{-\sigma}U\|_{L^2}\Big(\|\D_hV\|_{L^2}^{1-\sigma}\|\partial_3V\|_{L^2}^{\frac{1}{2}}\|\D_hU\|_{L^2}+\|\D_hV\|_{L^2}^{\frac{3}{2}-\sigma}\|\partial_3U\|_{L^2}\Big).
\end{align}

Applying $\Lambda_h^{-\sigma}\partial_3$ to \eqref{BMHD}, and taking the $L^2$-inner product with $(\Lambda_h^{-\sigma}\partial_3u, \Lambda_h^{-\sigma}\partial_3B, \Lambda_h^{-\sigma}\partial_3w)$, we obtain
\begin{align}\label{6.6}
&\frac{1}{2}\frac{\dif}{\dif t}\|(\Lambda_h^{-\sigma}\partial_3u, \Lambda_h^{-\sigma}\partial_3B, \Lambda_h^{-\sigma}\partial_3w)\|_{L^2}^2+\mu\|\Lambda_h^{1-\sigma}\partial_3u\|_{L^2}^2+\nu\|\Lambda_h^{1-\sigma}\partial_3B\|_{L^2}^2+\gamma\|\Lambda_h^{1-\sigma}\partial_3w\|_{L^2}^2 \nonumber\\
&+\chi\|\D\Lambda_h^{-\sigma}\partial_3u\|_{L^2}^2+4\chi\|\Lambda_h^{-\sigma}\partial_3w\|_{L^2}^2 \nonumber\\
\leq~&4\chi\int\Lambda_h^{-\sigma}\partial_3w\cdot\Lambda_h^{-\sigma}\partial_3(\D\times u) \dif x-\int\Lambda_h^{-\sigma}\partial_3(u\cdot\D u)\cdot\Lambda_h^{-\sigma}u\dif x+\int\Lambda_h^{-\sigma}\partial_3(B\cdot\D B)\cdot\Lambda_h^{-\sigma}u\dif x \nonumber\\
&-\int\Lambda_h^{-\sigma}\partial_3(u\cdot\D B)\cdot\Lambda_h^{-\sigma}B\dif x+\int\Lambda_h^{-\sigma}\partial_3(B\cdot\D u)\cdot\Lambda_h^{-\sigma}B\dif x-\int\Lambda_h^{-\sigma}\partial_3(u\cdot\D w)\cdot\Lambda_h^{-\sigma}w\dif x \nonumber\\
\leq~&\chi\|\D\Lambda_h^{-\sigma}\partial_3u\|_{L^2}^2+4\chi\|\Lambda_h^{-\sigma}\partial_3w\|_{L^2}^2-\int\Lambda_h^{-\sigma}\partial_3(u\cdot\D u)\cdot\Lambda_h^{-\sigma}\partial_3u\dif x+\int\Lambda_h^{-\sigma}\partial_3(B\cdot\D B)\cdot\Lambda_h^{-\sigma}\partial_3u\dif x \nonumber\\
&-\int\Lambda_h^{-\sigma}\partial_3(u\cdot\D B)\cdot\Lambda_h^{-\sigma}\partial_3B\dif x+\int\Lambda_h^{-\sigma}\partial_3(B\cdot\D u)\cdot\Lambda_h^{-\sigma}\partial_3B\dif x-\int\Lambda_h^{-\sigma}\partial_3(u\cdot\D w)\cdot\Lambda_h^{-\sigma}\partial_3w\dif x \nonumber\\
:=~&\chi\|\D\Lambda_h^{-\sigma}\partial_3u\|_{L^2}^2+4\chi\|\Lambda_h^{-\sigma}\partial_3w\|_{L^2}^2+\sum\limits_{i=1}^5M_i.
\end{align}

Noting that $\nabla\cdot u=0$, we have
\begin{align}\label{6.7}
M_1=~&-\int\Lambda_h^{-\sigma}(\partial_3u\cdot\D u)\cdot\Lambda_h^{-\sigma}\partial_3u\dif x-\int\Lambda_h^{-\sigma}(u\cdot\D\partial_3 u)\cdot\Lambda_h^{-\sigma}\partial_3u\dif x \nonumber\\
=~&-\int\Lambda_h^{-\sigma}(\partial_3u_h\cdot\D_h u)\cdot\Lambda_h^{-\sigma}\partial_3u\dif x-\int\Lambda_h^{-\sigma}(\partial_3u_3\partial_3 u)\cdot\Lambda_h^{-\sigma}\partial_3u\dif x \nonumber\\
&-\int\Lambda_h^{-\sigma}(u_h\cdot\D_h\partial_3 u)\cdot\Lambda_h^{-\sigma}\partial_3u\dif x-\int\Lambda_h^{-\sigma}(u_3\partial_3^2 u)\cdot\Lambda_h^{-\sigma}\partial_3u\dif x \nonumber\\
=~&-\int\Lambda_h^{-\sigma}(\partial_3u_h\cdot\D_h u)\cdot\Lambda_h^{-\sigma}\partial_3u\dif x+\int\Lambda_h^{-\sigma}(\D_h\cdot u_h\partial_3 u)\cdot\Lambda_h^{-\sigma}\partial_3u\dif x \nonumber\\
&-\int\Lambda_h^{-\sigma}(u_h\cdot\D_h\partial_3 u)\cdot\Lambda_h^{-\sigma}\partial_3u\dif x-\int\Lambda_h^{-\sigma}(u_3\partial_3^2 u)\cdot\Lambda_h^{-\sigma}\partial_3u\dif x \nonumber\\
:=~&\sum\limits_{j=1}^4M_{1,j}.
\end{align}

Similar to \eqref{6.2}, we have
\begin{align}\label{6.8}
M_{1,1}+M_{1,2}\leq~&C\|\Lambda_h^{-\sigma}\partial_3u\|_{L^2}\Big(\big\|\Lambda_h^{-\sigma}(\partial_3u_h\cdot\D_h u)\big\|_{L^2}+\big\|\Lambda_h^{-\sigma}(\partial_3 u\D_h\cdot u_h)\big\|_{L^2}\Big) \nonumber\\
\leq~&C\|\Lambda_h^{-\sigma}\partial_3u\|_{L^2}\|\partial_3u\|_{L^2}^{\sigma-\frac{1}{2}}\|\D_h\partial_3u_h\|_{L^2}^{1-\sigma}\|\partial_3^2u\|_{L^2}^{\frac{1}{2}}\|\D_hu\|_{L^2} \nonumber\\
\leq~&C\|\Lambda_h^{-\sigma}\partial_3u\|_{L^2}\|\partial_3u\|_{L^2}^{\sigma-\frac{1}{2(k-1)}}\|\D_hu\|_{L^2}^{1+\frac{(k-2)(1-\sigma)}{k-1}}\|\D_h\partial_3^{k-1}u\|_{L^2}^{\frac{1-\sigma}{k-1}}\|\partial_3^ku\|_{L^2}^{\frac{1}{2(k-1)}},
\end{align}
and
\begin{align}\label{6.9}
M_{1,3}+M_{1,4}\leq~&C\|\Lambda_h^{-\sigma}\partial_3u\|_{L^2}\Big(\big\|\Lambda_h^{-\sigma}(u_h\cdot\D_h\partial_3 u)\big\|_{L^2}+\big\|\Lambda_h^{-\sigma}(u_3\partial_3^2 u)\big\|_{L^2}\Big) \nonumber\\
\leq~&C\|\Lambda_h^{-\sigma}\partial_3u\|_{L^2}\|u\|_{L^2}^{\sigma-\frac{1}{2}}\|\D_hu\|_{L^2}^{1-\sigma}\Big(\|\partial_3u\|_{L^2}^{\frac{1}{2}}\|\D_h\partial_3u\|_{L^2}+\|\partial_3u_3\|_{L^2}^{\frac{1}{2}}\|\partial_3^2u\|_{L^2}\Big) \nonumber\\
\leq~&C\|\Lambda_h^{-\sigma}\partial_3u\|_{L^2}\|u\|_{L^2}^{\sigma-\frac{1}{2}}\|\D_hu\|_{L^2}^{1-\sigma}\|\partial_3u\|_{L^2}^{\frac{1}{2}}\|\D_hu\|_{L^2}^{\frac{k-2}{k-1}}\|\D_h\partial_3^{k-1}u\|_{L^2}^{\frac{1}{k-1}} \nonumber\\
&+C\|\Lambda_h^{-\sigma}\partial_3u\|_{L^2}\|u\|_{L^2}^{\sigma-\frac{1}{2}}\|\D_hu\|_{L^2}^{\frac{3}{2}-\sigma}\|\partial_3u\|_{L^2}^{\frac{k-2}{k-1}}\|\partial_3^ku\|_{L^2}^{\frac{1}{k-1}},
\end{align}
where we have used the following estimates that come from Gagliardo-Nirenberg inequality,
\begin{align*}
\|\D_h\partial_3u\|_{L^2}\leq~&C\|\D_hu\|_{L^2}^{\frac{k-2}{k-1}}\|\D_h\partial_3^{k-1}u\|_{L^2}^{\frac{1}{k-1}}, \\
\|\partial_3^2u\|_{L^2}\leq~&C\|\partial_3u\|_{L^2}^{\frac{k-2}{k-1}}\|\partial_3^ku\|_{L^2}^{\frac{1}{k-1}}.
\end{align*}
Inserting \eqref{6.8} and \eqref{6.9} into \eqref{6.7}, we obtain
\begin{align}\label{6.10}
M_1\leq~&C\|\Lambda_h^{-\sigma}\partial_3u\|_{L^2}\|\partial_3u\|_{L^2}^{\sigma-\frac{1}{2(k-1)}}\|\D_hu\|_{L^2}^{1+\frac{(k-2)(1-\sigma)}{k-1}}\|\D_h\partial_3^{k-1}u\|_{L^2}^{\frac{1-\sigma}{k-1}}\|\partial_3^ku\|_{L^2}^{\frac{1}{2(k-1)}} \nonumber\\
&+C\|\Lambda_h^{-\sigma}\partial_3u\|_{L^2}\|u\|_{L^2}^{\sigma-\frac{1}{2}}\|\D_hu\|_{L^2}^{1-\sigma}\|\partial_3u\|_{L^2}^{\frac{1}{2}}\|\D_hu\|_{L^2}^{\frac{k-2}{k-1}}\|\D_h\partial_3^{k-1}u\|_{L^2}^{\frac{1}{k-1}} \nonumber\\
&+C\|\Lambda_h^{-\sigma}\partial_3u\|_{L^2}\|u\|_{L^2}^{\sigma-\frac{1}{2}}\|\D_hu\|_{L^2}^{\frac{3}{2}-\sigma}\|\partial_3u\|_{L^2}^{\frac{k-2}{k-1}}\|\partial_3^ku\|_{L^2}^{\frac{1}{k-1}}.
\end{align}

Similar to \eqref{6.7}--\eqref{6.10}, we have
\begin{align}\label{6.11}
M_3+M_5\leq~&C\|\Lambda_h^{-\sigma}\partial_3B\|_{L^2}\Big(\big\|\Lambda_h^{-\sigma}(\partial_3u_h\cdot\D_h B)\big\|_{L^2}+\big\|\Lambda_h^{-\sigma}(\partial_3 B\D_h\cdot u_h)\big\|_{L^2}\Big) \nonumber\\
&+C\|\Lambda_h^{-\sigma}\partial_3B\|_{L^2}\Big(\big\|\Lambda_h^{-\sigma}(u_h\cdot\D_h\partial_3 B)\big\|_{L^2}+\big\|\Lambda_h^{-\sigma}(u_3\partial_3^2 B)\big\|_{L^2}\Big) 
\nonumber\\
&+C\|\Lambda_h^{-\sigma}\partial_3w\|_{L^2}\Big(\big\|\Lambda_h^{-\sigma}(\partial_3u_h\cdot\D_h w)\big\|_{L^2}+\big\|\Lambda_h^{-\sigma}(\partial_3 w\D_h\cdot u_h)\big\|_{L^2}\Big) \nonumber\\
&+C\|\Lambda_h^{-\sigma}\partial_3w\|_{L^2}\Big(\big\|\Lambda_h^{-\sigma}(u_h\cdot\D_h\partial_3 w)\big\|_{L^2}+\big\|\Lambda_h^{-\sigma}(u_3\partial_3^2 w)\big\|_{L^2}\Big) \nonumber\\
\leq~&C\|\Lambda_h^{-\sigma}\partial_3B\|_{L^2}\|\partial_3u\|_{L^2}^{\sigma-\frac{1}{2(k-1)}}\|\D_hu\|_{L^2}^{\frac{(k-2)(1-\sigma)}{k-1}}\|\D_h\partial_3^{k-1}u\|_{L^2}^{\frac{1-\sigma}{k-1}}\|\partial_3^ku\|_{L^2}^{\frac{1}{2(k-1)}}\|\D_hB\|_{L^2} \nonumber\\
&+C\|\Lambda_h^{-\sigma}\partial_3B\|_{L^2}\|\partial_3B\|_{L^2}^{\sigma-\frac{1}{2(k-1)}}\|\D_hB\|_{L^2}^{\frac{(k-2)(1-\sigma)}{k-1}}\|\D_h\partial_3^{k-1}B\|_{L^2}^{\frac{1-\sigma}{k-1}}\|\partial_3^kB\|_{L^2}^{\frac{1}{2(k-1)}}\|\D_hu\|_{L^2} \nonumber\\
&+C\|\Lambda_h^{-\sigma}\partial_3B\|_{L^2}\|u\|_{L^2}^{\sigma-\frac{1}{2}}\|\D_hu\|_{L^2}^{1-\sigma}\|\partial_3u\|_{L^2}^{\frac{1}{2}}\|\D_hB\|_{L^2}^{\frac{k-2}{k-1}}\|\D_h\partial_3^{k-1}B\|_{L^2}^{\frac{1}{k-1}} \nonumber\\
&+C\|\Lambda_h^{-\sigma}\partial_3B\|_{L^2}\|u\|_{L^2}^{\sigma-\frac{1}{2}}\|\D_hu\|_{L^2}^{\frac{3}{2}-\sigma}\|\partial_3B\|_{L^2}^{\frac{k-2}{k-1}}\|\partial_3^kB\|_{L^2}^{\frac{1}{k-1}} \nonumber\\
&+C\|\Lambda_h^{-\sigma}\partial_3w\|_{L^2}\|\partial_3u\|_{L^2}^{\sigma-\frac{1}{2(k-1)}}\|\D_hu\|_{L^2}^{\frac{(k-2)(1-\sigma)}{k-1}}\|\D_h\partial_3^{k-1}u\|_{L^2}^{\frac{1-\sigma}{k-1}}\|\partial_3^ku\|_{L^2}^{\frac{1}{2(k-1)}}\|\D_hw\|_{L^2} \nonumber\\
&+C\|\Lambda_h^{-\sigma}\partial_3w\|_{L^2}\|\partial_3w\|_{L^2}^{\sigma-\frac{1}{2(k-1)}}\|\D_hw\|_{L^2}^{\frac{(k-2)(1-\sigma)}{k-1}}\|\D_h\partial_3^{k-1}w\|_{L^2}^{\frac{1-\sigma}{k-1}}\|\partial_3^kw\|_{L^2}^{\frac{1}{2(k-1)}}\|\D_hu\|_{L^2} \nonumber\\
&+C\|\Lambda_h^{-\sigma}\partial_3w\|_{L^2}\|u\|_{L^2}^{\sigma-\frac{1}{2}}\|\D_hu\|_{L^2}^{1-\sigma}\|\partial_3u\|_{L^2}^{\frac{1}{2}}\|\D_hw\|_{L^2}^{\frac{k-2}{k-1}}\|\D_h\partial_3^{k-1}w\|_{L^2}^{\frac{1}{k-1}} \nonumber\\
&+C\|\Lambda_h^{-\sigma}\partial_3w\|_{L^2}\|u\|_{L^2}^{\sigma-\frac{1}{2}}\|\D_hu\|_{L^2}^{\frac{3}{2}-\sigma}\|\partial_3w\|_{L^2}^{\frac{k-2}{k-1}}\|\partial_3^kw\|_{L^2}^{\frac{1}{k-1}}.
\end{align}

Noting that $\nabla\cdot B=0$, similar to \eqref{6.7}--\eqref{6.10}, we have
\begin{align}\label{6.12}
M_2=~&\int\Lambda_h^{-\sigma}(\partial_3B\cdot\D B)\cdot\Lambda_h^{-\sigma}\partial_3u\dif x+\int\Lambda_h^{-\sigma}(B\cdot\D\partial_3 B)\cdot\Lambda_h^{-\sigma}\partial_3u\dif x \nonumber\\
=~&\int\Lambda_h^{-\sigma}(\partial_3B_h\cdot\D_h B)\cdot\Lambda_h^{-\sigma}\partial_3u\dif x+\int\Lambda_h^{-\sigma}(\partial_3B_3\partial_3 B)\cdot\Lambda_h^{-\sigma}\partial_3u\dif x \nonumber\\
&+\int\Lambda_h^{-\sigma}(B_h\cdot\D_h\partial_3 B)\cdot\Lambda_h^{-\sigma}\partial_3u\dif x+\int\Lambda_h^{-\sigma}(B_3\partial_3^2 B)\cdot\Lambda_h^{-\sigma}\partial_3u\dif x \nonumber\\
=~&\int\Lambda_h^{-\sigma}(\partial_3B_h\cdot\D_h B)\cdot\Lambda_h^{-\sigma}\partial_3u\dif x-\int\Lambda_h^{-\sigma}(\D_h\cdot B_h\partial_3 B)\cdot\Lambda_h^{-\sigma}\partial_3u\dif x \nonumber\\
&+\int\Lambda_h^{-\sigma}(B_h\cdot\D_h\partial_3 B)\cdot\Lambda_h^{-\sigma}\partial_3u\dif x+\int\Lambda_h^{-\sigma}(B_3\partial_3^2 B)\cdot\Lambda_h^{-\sigma}\partial_3u\dif x \nonumber\\
\leq~&C\|\Lambda_h^{-\sigma}\partial_3u\|_{L^2}\|\partial_3B\|_{L^2}^{\sigma-\frac{1}{2(k-1)}}\|\D_hB\|_{L^2}^{1+\frac{(k-2)(1-\sigma)}{k-1}}\|\D_h\partial_3^{k-1}B\|_{L^2}^{\frac{1-\sigma}{k-1}}\|\partial_3^kB\|_{L^2}^{\frac{1}{2(k-1)}} \nonumber\\
&+C\|\Lambda_h^{-\sigma}\partial_3u\|_{L^2}\|B\|_{L^2}^{\sigma-\frac{1}{2}}\|\D_hB\|_{L^2}^{1-\sigma}\|\partial_3B\|_{L^2}^{\frac{1}{2}}\|\D_hB\|_{L^2}^{\frac{k-2}{k-1}}\|\D_h\partial_3^{k-1}B\|_{L^2}^{\frac{1}{k-1}} \nonumber\\
&+C\|\Lambda_h^{-\sigma}\partial_3u\|_{L^2}\|B\|_{L^2}^{\sigma-\frac{1}{2}}\|\D_hB\|_{L^2}^{\frac{3}{2}-\sigma}\|\partial_3B\|_{L^2}^{\frac{k-2}{k-1}}\|\partial_3^kB\|_{L^2}^{\frac{1}{k-1}},
\end{align}
and
\begin{align}\label{6.13}
M_4=~&\int\Lambda_h^{-\sigma}(\partial_3B\cdot\D u)\cdot\Lambda_h^{-\sigma}\partial_3B\dif x+\int\Lambda_h^{-\sigma}(B\cdot\D\partial_3 u)\cdot\Lambda_h^{-\sigma}\partial_3B\dif x \nonumber\\
=~&\int\Lambda_h^{-\sigma}(\partial_3B_h\cdot\D_h u)\cdot\Lambda_h^{-\sigma}\partial_3B\dif x+\int\Lambda_h^{-\sigma}(\partial_3B_3\partial_3 u)\cdot\Lambda_h^{-\sigma}\partial_3B\dif x \nonumber\\
&+\int\Lambda_h^{-\sigma}(B_h\cdot\D_h\partial_3 u)\cdot\Lambda_h^{-\sigma}\partial_3B\dif x+\int\Lambda_h^{-\sigma}(B_3\partial_3^2 u)\cdot\Lambda_h^{-\sigma}\partial_3B\dif x \nonumber\\
=~&\int\Lambda_h^{-\sigma}(\partial_3B_h\cdot\D_h u)\cdot\Lambda_h^{-\sigma}\partial_3B\dif x-\int\Lambda_h^{-\sigma}(\D_h\cdot B_h\partial_3 u)\cdot\Lambda_h^{-\sigma}\partial_3B\dif x \nonumber\\
&+\int\Lambda_h^{-\sigma}(B_h\cdot\D_h\partial_3 u)\cdot\Lambda_h^{-\sigma}\partial_3B\dif x+\int\Lambda_h^{-\sigma}(B_3\partial_3^2 u)\cdot\Lambda_h^{-\sigma}\partial_3B\dif x \nonumber\\
\leq~&C\|\Lambda_h^{-\sigma}\partial_3B\|_{L^2}\|\partial_3B\|_{L^2}^{\sigma-\frac{1}{2(k-1)}}\|\D_hB\|_{L^2}^{\frac{(k-2)(1-\sigma)}{k-1}}\|\D_h\partial_3^{k-1}B\|_{L^2}^{\frac{1-\sigma}{k-1}}\|\partial_3^kB\|_{L^2}^{\frac{1}{2(k-1)}}\|\D_hu\|_{L^2} \nonumber\\
&+C\|\Lambda_h^{-\sigma}\partial_3B\|_{L^2}\|\partial_3u\|_{L^2}^{\sigma-\frac{1}{2(k-1)}}\|\D_hu\|_{L^2}^{\frac{(k-2)(1-\sigma)}{k-1}}\|\D_h\partial_3^{k-1}u\|_{L^2}^{\frac{1-\sigma}{k-1}}\|\partial_3^ku\|_{L^2}^{\frac{1}{2(k-1)}}\|\D_hB\|_{L^2} \nonumber\\
&+C\|\Lambda_h^{-\sigma}\partial_3B\|_{L^2}\|B\|_{L^2}^{\sigma-\frac{1}{2}}\|\D_hB\|_{L^2}^{1-\sigma}\|\partial_3B\|_{L^2}^{\frac{1}{2}}\|\D_hu\|_{L^2}^{\frac{k-2}{k-1}}\|\D_h\partial_3^{k-1}u\|_{L^2}^{\frac{1}{k-1}} \nonumber\\
&+C\|\Lambda_h^{-\sigma}\partial_3B\|_{L^2}\|B\|_{L^2}^{\sigma-\frac{1}{2}}\|\D_hB\|_{L^2}^{\frac{3}{2}-\sigma}\|\partial_3u\|_{L^2}^{\frac{k-2}{k-1}}\|\partial_3^ku\|_{L^2}^{\frac{1}{k-1}}.
\end{align}

Substituting \eqref{6.10}--\eqref{6.13} into \eqref{6.6} and recalling \eqref{UV}, we obtain
\begin{align}\label{6.14}
&\frac{\dif}{\dif t}\|\Lambda_h^{-\sigma}\partial_3U\|_{L^2}^2+C_0\|\Lambda_h^{1-\sigma}\partial_3U\|_{L^2}^2 \nonumber\\
\leq~&C\|\Lambda_h^{-\sigma}\partial_3U\|_{L^2}\|\partial_3U\|_{L^2}^{\sigma-\frac{1}{2(k-1)}}\|\D_hU\|_{L^2}^{\frac{(k-2)(1-\sigma)}{k-1}}\|\D_h\partial_3^{k-1}U\|_{L^2}^{\frac{1-\sigma}{k-1}}\|\partial_3^kU\|_{L^2}^{\frac{1}{2(k-1)}}\|\D_hU\|_{L^2} \nonumber\\
&+C\|\Lambda_h^{-\sigma}\partial_3U\|_{L^2}\|V\|_{L^2}^{\sigma-\frac{1}{2}}\|\D_hV\|_{L^2}^{1-\sigma}\|\partial_3V\|_{L^2}^{\frac{1}{2}}\|\D_hU\|_{L^2}^{\frac{k-2}{k-1}}\|\D_h\partial_3^{k-1}U\|_{L^2}^{\frac{1}{k-1}} \nonumber\\
&+C\|\Lambda_h^{-\sigma}\partial_3U\|_{L^2}\|V\|_{L^2}^{\sigma-\frac{1}{2}}\|\D_hV\|_{L^2}^{\frac{3}{2}-\sigma}\|\partial_3U\|_{L^2}^{\frac{k-2}{k-1}}\|\partial_3^kU\|_{L^2}^{\frac{1}{k-1}}.
\end{align}

By Theorem \ref{th}, Lemma \ref{lem.4}, Lemma \ref{lem.5}, and \eqref{3.1}, it holds that
\begin{align}\label{7.0}
\|\Lambda_h^{-\sigma}U\|_{L^2}+\|\Lambda_h^{-\sigma}\partial_3U\|_{L^2}+\|U\|_{H^k}\leq~& C, \nonumber\\
\|U\|_{L^2}+\|\partial_3U\|_{L^2}\leq~&C\varepsilon, \nonumber\\
\|U\|_{L^2}+\|\partial_3U\|_{L^2}\leq~&C(1+t)^{-\frac{\sigma}{2}}, \nonumber\\
\|\D_hU\|_{L^2}\leq~&C(1+t)^{-\frac{1+\sigma}{2}}.
\end{align}
Adding \eqref{6.5'} to \eqref{6.14} and integrating the resulting inequality on $(0,t)$ for some $t\in(0,T]$, we obtain
\begin{align}\label{7.1}
&\|\Lambda_h^{-\sigma}U\|_{L^2}^2+\|\Lambda_h^{-\sigma}\partial_3U\|_{L^2}^2 \nonumber\\
\leq~&\|\Lambda_h^{-\sigma}U(0)\|_{L^2}^2+\|\Lambda_h^{-\sigma}\partial_3U(0)\|_{L^2}^2+C\varepsilon^\delta\int_0^t(1+\tau)^{-\frac{3}{4}-\frac{\sigma}{2}(1-\delta)}\dif\tau \nonumber\\
&+C\varepsilon^\delta\int_0^t(1+\tau)^{-1-\frac{\sigma}{2}(1-\delta)-\frac{2\sigma^2-\sigma-2}{4(k-1)}}\dif\tau+C\varepsilon^\delta\int_0^t(1+\tau)^{-1-\frac{\sigma}{2}(1-\delta)-\frac{1+\sigma}{2(k-1)}}\dif\tau \nonumber\\
&+C\varepsilon^\delta\int_0^t(1+\tau)^{-\frac{3}{4}-\frac{\sigma}{2}(1-\delta)+\frac{\sigma}{2(k-1)}}\dif\tau,
\end{align}
for fixed $\delta>0$. Noting that $\frac{k-1}{2(k-2)}<\sigma<1$ for some $k\geq 4$, we can choose $\delta$ sufficiently small, such that
\begin{align*}
-\frac{3}{4}-\frac{\sigma}{2}(1-\delta)<-1, \quad & \quad
-1-\frac{\sigma}{2}(1-\delta)-\frac{2\sigma^2-\sigma-2}{4(k-1)}<-1, \\
-1-\frac{\sigma}{2}(1-\delta)-\frac{2\sigma^2-\sigma-2}{4(k-1)}<-1, \quad & \quad
-\frac{3}{4}-\frac{\sigma}{2}(1-\delta)+\frac{\sigma}{2(k-1)}<-1.
\end{align*}
Hence we get
\begin{equation*}
\|\Lambda_h^{-\sigma}U\|_{L^2}^2+\|\Lambda_h^{-\sigma}\partial_3U\|_{L^2}^2\leq\|\Lambda_h^{-\sigma}U(0)\|_{L^2}^2+\|\Lambda_h^{-\sigma}\partial_3U(0)\|_{L^2}^2+C\varepsilon^\delta\leq E_0+C\varepsilon^\delta,
\end{equation*}
for some positive constant $C$ independent of $\varepsilon$. Taking $\varepsilon$ sufficiently small, we complete the proof of \eqref{6.0}. At last, through a standard bootstrap argument, we conclude that \eqref{3.1} holds for any $t\geq 0$. Thus the decay estimates \eqref{4.0} and \eqref{5.0} hold for any $t\geq 0$. This completes the proof of Theorem \ref{th2}.
\end{proof}

\bigskip

\noindent{\bf Acknowledgments} \\ 
This work is supported by Science Foundation of Zhejiang Sci-Tech University (ZSTU) under Grant No. 25062122-Y.

\bigskip 
 
\noindent{\bf Data Availability Statements} \\ 
Data sharing not applicable to this article as no datasets were generated or analyzed during the current study.

\bigskip

\noindent{\bf Conflict of interests} \\
The authors declare that they have no competing interests.

\bigskip

\noindent{\bf Authors' contributions} \\
The authors have made the same contribution. All authors read and approved the final manuscript.

\bibliographystyle{plain}

\end{document}